\documentclass[11pt,reqno]{amsart}    
\usepackage[margin=1in]{geometry}
\usepackage{amsfonts,amsmath,amssymb,wasysym,stmaryrd,enumerate}    
\usepackage{amsthm}
\usepackage{mathrsfs}  
\usepackage{setspace}
\usepackage{bm}
\onehalfspace
\usepackage{tikz}    
\usepackage{ytableau}
\usepackage{comment}
\ytableausetup{mathmode, boxsize=2.5mm}
\usetikzlibrary{arrows,matrix,decorations.markings,shapes.geometric}    

\numberwithin{equation}{section}    

\theoremstyle{plain}
\newtheorem{thm}{Theorem}[section]
\newtheorem{lem}[thm]{Lemma}
\newtheorem{prop}[thm]{Proposition}

\theoremstyle{definition}

\newtheorem{exmp}[thm]{Example}
\theoremstyle{remark}
\newtheorem{rem}[thm]{Remark}

\newtheorem*{rem*}{Remark}
\newtheorem*{ack}{Acknowledgments}

\newcommand{\bs}{\boldsymbol}
    
\newcommand{\be}{\begin{equation}}    
\newcommand{\ee}{\end{equation}}    
\newcommand{\beu}{\begin{equation*}}    
\newcommand{\eeu}{\end{equation*}}    
\newcommand{\bea}{\begin{eqnarray}}    
\newcommand{\eea}{\end{eqnarray}}    
\newcommand{\beaa}{\begin{eqnarray*}}    
\newcommand{\eeaa}{\end{eqnarray*}}    
\newcommand{\bmx}{\begin{pmatrix}}    
\newcommand{\emx}{\end{pmatrix}}

\newcommand{\del}{\partial}    
\newcommand{\g}{{\mathfrak g}}    
\newcommand{\h}{{\mathfrak h}}

\newcommand{\p}{{\mathfrak p}}

\newcommand{\mf}{\mathfrak}
\newcommand{\mc}{\mathcal}

\newcommand{\half}{\frac{1}{2}}

\newcommand{\nn}{\nonumber}

\newcommand{\8}{{\infty}}

\newcommand{\eps}{\epsilon}    
\newcommand{\vareps}{\varepsilon}

\newcommand{\Z}{{\mathbb Z}}

\newcommand{\C}{{\mathbb C}}    
    
\newcommand{\PP}{\mathbb{P}}

\newcommand{\id}{{\mathrm{id}}}

\newcommand{\goi}[2]{=}    
\newcommand{\Hom}{\mathrm{Hom}}

\newcommand{\on}{}    

\newcommand{\Cx}{\mathbb C^*}

\newcommand{\btp}{\begin{tikzpicture}[baseline=0pt,scale=0.9,line width=0.25pt]}    
\newcommand{\etp}{\end{tikzpicture}}

\newcommand{\pr}[1]{{\left| #1 \right|}}

\newcommand{\atp}[1]{}

\newcommand{\path}{\longrightarrow}

\newcommand{\gl}{\mf{gl}}
\renewcommand{\sl}{\mf{sl}}

\DeclareMathOperator{\End}{End}

\DeclareMathOperator{\str}{sTr}
\DeclareMathOperator{\ind}{Ind}

\newcommand{\HH}{\mathcal H}

\begin{document}

\title{Gaudin models for $\mf{gl}(m|n)$}

\author{Evgeny Mukhin} 
\email{mukhin@math.iupui.edu}
\address{Department of Mathematical Sciences, 402
  N. Blackford St, LD 270, IUPUI, Indianapolis, IN 46202, USA.  }

\author{Beno\^{\i}t Vicedo}
\email{benoit.vicedo@gmail.com}
\author{Charles Young}
\email{charlesyoung@cantab.net}
\address{School of Physics, Astronomy and Mathematics, University of Hertfordshire, College Lane, Hatfield AL10 9AB, UK.}  

\begin{abstract}
We establish the basics of the Bethe ansatz for the Gaudin model associated to the Lie superalgebra $\gl(m|n)$.
In particular, we prove the completeness of the Bethe ansatz in the case of tensor products of fundamental representations.
\end{abstract}

\maketitle
\setcounter{tocdepth}{1}
\tableofcontents

\section{Introduction}
The Gaudin model \cite{G,Gbook} is a famous quantum integrable model extensively studied in the last quarter of a century by many mathematicians using various methods. In recent years, the theory of the Bethe ansatz in relation to the Gaudin model has enjoyed a new wave of popularity within the mathematical community; see for example \cite{Fre}, \cite{Freview}, \cite{Rybnikov}, \cite{Talalaev}, \cite {FFRb}, \cite{FFTL}, \cite{MVcrit}, \cite{MTV2}. 

In particular, several important results have been obtained in type $A$. These include the proof of the Shapiro-Shapiro conjecture related to Grassmannians, the proof of the simplicity of the spectrum of the model, and the Langlands-type identification of Bethe vectors and Fuchsian differential operators without monodromy; see \cite{MTV2}. However, attempts to obtain similar results in other Lie types or in XXX models have not been fully successful so far.  

In this paper we start the study 
of another model, namely the generalization of the Gaudin model of type $A$ to the 
case of Lie superalgebras $\gl(m|n)$. In some respects this model seems very similar to the even case of $\gl(n)$, while in others it is very different. 

Supersymmetric versions of the models of mathematical physics have been intensively studied in the literature. It is quite surprising that the Gaudin model associated to $\gl(m|n)$ has been largely omitted from these studies so far. Therefore, we are forced to start at the very beginning.

In this paper we define the weight function (or the wave function), which is similar to the even case of \cite{SV}. We introduce the Bethe ansatz equations (BAE) and show that when the parameters of the weight function satisfy the BAE, then its value is a singular vector which is an eigenvector of the quadratic Hamiltonians with an explicit eigenvalue. In the even case such a result is normally proved either from the corresponding result in the XXX model (where the algebraic Bethe ansatz method is used \cite{MTV1}), by methods of conformal field theory \cite{FFR}, or by studying of asymptotics of hypergeometric solutions of the Knizhnik-Zamolodchikov equation \cite{RV}. In the existing literature, the algebraic Bethe ansatz method has been  studied for supersymmetric XXX-type models \cite{BelliardRagoucy}. Also, a description of the Feigin-Frenkel centre of the vacuum Verma module at critical level is available in type $\gl(m|n)$ \cite{MolevRagoucy}. Here, however, we prefer a direct approach. 
Indeed, we find 
that the computation is fairly explicit and direct for the case of tensor products of vector representations, once one uses induction on the number of auxiliary variables (in contrast to \cite{BabujianFlume} where induction on the rank of the algebra was suggested). See Theorem \ref{Hthm}. The general case is then obtained by analytic arguments in Theorem \ref{ee2}.

We then proceed to study the case of tensor products of vector representations. We solve the case of two points and proceed to use functorial properties of the construction to show that the eigenvectors given by the Bethe ansatz form a basis in the space of singular vectors for generic values of the parameters. See Theorem \ref{compthm}. Our methods are similar to those of \cite{MV1}, \cite{MV2}.

In sharp contrast to the even case, the BAE depend on the choice of root system in $\gl(m|n)$. Therefore, one has an interplay between solutions of seemingly different sets of equations. This interplay can be expected to play an important role in the description of the supersymmetric populations of Bethe ansatz solutions -- cf. \cite{MVcrit} for the even case -- and we intend to clarify this point in further publications.

The paper is constructed as follows. In Section \ref{back sec} we give basic definitions and conventions. In Section \ref{Ham sec} we define the supersymmetric Gaudin model and establish some basic properties. In Section \ref{BA sec} we describe the Bethe ansatz method and solve the 2 point case. In Section \ref{func sec} we use the results of the previous section to establish the completeness of the Bethe ansatz method for tensor products of vector representations. In Section \ref{other sec} we discuss explicitly the simplest case of $\gl(1|1)$ and then give examples in the case of $\gl(2|1)$.

\begin{ack} EM would like to thank the University of Hertfordshire for hospitality during his visit, during which the majority of this work was accomplished.
\end{ack}

\section{Background}\label{back sec}
In this section we recall some standard facts concerning the Lie superalgebra $\gl(m|n)$ and its representation theory. For more details the reader is referred to, for example, \cite{ChengWang} and \cite{Dictionary}.
\subsection{Conventions on superspaces and superalgebras}
We work over $\C$. A \emph{vector superspace} $V=V_{\bar 0}\oplus V_{\bar 1}$ is a vector space with a $\Z_2$-gradation. Elements of $V_{\bar 0}$ are called \emph{even}; elements of $V_{\bar 1}$ are called \emph{odd}. We write $\pr v\in \{0,1\}$ for the parity of a homogeneous element $v\in V$. The direct sum $V\oplus W$ and the tensor product $V\otimes W$ of two vector superspaces $V$ and $W$ are vector superspaces with 
$(V\oplus W)_{\bar 0} := V_{\bar 0} \oplus W_{\bar 0}$, $(V\oplus W)_{\bar 1} := V_{\bar 1} \oplus W_{\bar 1}$,
$(V\otimes W)_{\bar 0} := V_{\bar 0} \otimes W_{\bar 0} \oplus V_{\bar 1} \otimes W_{\bar 1}$ and $(V\otimes W)_{\bar 1} := V_{\bar 0} \otimes W_{\bar 1} \oplus V_{\bar 1} \otimes W_{\bar 0}$.

Linear maps between superspaces that preserve parity are called \emph{even}; those that flip parity are called \emph{odd}.
Then, for any finite-dimensional superspaces $V$ and $W$, the set of all linear maps $\Hom(V,W)$ is a superspace.

We use the usual super convention: if $a: V_1 \to V_2$ and $b: W_1\to W_2$ are odd or even linear maps between vector superspaces, then we define the map $(a\otimes b) : V_1\otimes W_1 \to V_2\otimes W_2$ by
\be (a\otimes b)(v\otimes w) = (av\otimes bw) (-1)^{\pr b \pr v},\label{grotimes}\ee
with $v\in V_1$ and $w\in W_1$.  Here and throughout, when we write for example $\pr v$ we always implicitly assume $v$ is homogeneous and extend the formula in question by linearity.

An \emph{(associative) superalgebra} $A$ is a vector superspace equipped with an even associative product, that is $A_{\bar i}A_{\bar j}\subset A_{\overline{i+j}}$. The tensor product of superalgebras $A$ and $B$ is the superspace $A\otimes B$ with the product
\be
(a_1\otimes b_1)(a_2\otimes b_2)=(-1)^{\pr {b_1}\pr {a_2}} a_1a_2\otimes b_1b_2.
\nn\ee

An even (respectively odd) homomorphism of superalgebras is an algebra homomorphism which is an even (respectively odd) linear map. For any superspace $V$, the space of linear endomorphisms $\gl(V)$ is a superalgebra. A 
\emph{module} over (or equivalently a \emph{representation} of) an associative superalgebra $A$ is a superspace $V$ with an even homomorphism of superalgebras $A\to\gl(V)$.

A \emph{Lie superalgebra} $\mf a$ is a superspace equipped with a supercommutator, namely a bilinear product $[\cdot,\cdot]:\mf a\times \mf a \to \mf a$ that is super skew-symmetric and obeys the super Jacobi identity: 
\be  [X,Y]= -(-1)^{\pr X \pr Y}[Y,X] \quad\text{and}\quad [X,[Y,Z]] = [[X,Y],Z] + (-1)^{\pr X \pr Y} [Y,[X,Z]]\nn\ee for $X,Y,Z\in \mf a$.
The \emph{universal enveloping algebra} $U(\mf a)$ is the unital superalgebra obtained by taking the quotient of the tensor algebra $T(\mf a)$ by the two-sided ideal generated by
\be X \otimes Y -  (-1)^{\pr X\pr Y} Y\otimes X  - \left[X , Y \right] , \qquad X,Y\in \mf a.\nn\ee
The coproduct $\Delta: U(\mf a) \to U(\mf a) \otimes U(\mf a)$ is defined to be the even homomorphism of superalgebras such that $\Delta X = X \otimes 1 + 1\otimes X$, $X\in \mf a$.

\subsection{Definitions of $\C^{m|n}$ and $\gl(m|n)$}
Let $\C^{m|n}$ denote the complex vector superspace with $\dim(\C^{m|n}_{\bar 0})=m$ and $\dim(\C^{m|n}_{\bar 1})=n$. We choose a basis $e_a$, $1\leq a \leq m+n$, consisting of homogeneous vectors, $m$ of which are even and $n$ of which are odd. For brevity we shall write their parities as 
$\pr{e_a} = \pr a.$ There is a {\it distinguished} choice of the basis vectors, such that
\be\pr a = \begin{cases} 0 & 1\leq a\leq m, \\ 1 & m+1 \leq a \leq m+n.\end{cases}\label{dist}\ee
However, it is convenient to leave unspecified, for now, which basis vectors are even and which are odd, with the following exception: 
\be\text{we shall always assume that $e_1$ is even, i.e. $\pr 1=0$.}\label{1}\ee

The space $\End(\C^{m|n})$ is a Lie superalgebra with the product $[X,Y] := XY - (-1)^{\pr X \pr Y} YX$. This Lie superalgebra can be described as a Lie superalgebra generated by elements $E_{ab}$, $1\leq a,b\leq m+n$, with parity 
\be \pr{E_{ab}} := \pr a + \pr b\nn,\ee
and the supercommutator given by:
\be [E_{ab}, E_{cd}] = \delta_{bc} E_{ad} - (-1)^{(\pr a + \pr b)(\pr c + \pr d)} \delta_{ad} E_{cb}.\label{bdef}\ee
The element $E_{ab}$ corresponds to the linear operator $e_{ab}\in \End(\C^{m|n})$ such that
\be e_{ab} e_c = \delta_{bc} e_a.\nn\ee
We call this Lie superalgebra $\gl(m|n)$ and the space $\C^{m|n}$ its \emph{defining representation}.

Note that the Lie superalgebras $\gl(m|n)$ and $\gl(n|m)$ are isomorphic, as flipping the parity of all the $e_a$ does not change the commutator. Also, the Dynkin diagram -- see below -- is invariant under flipping the parity of all the $e_a$ as well. In particular, the convention (\ref{1}) does not lead to a loss of generality.

The \emph{supertrace} is the linear map $\str: \End(\C^{m|n})\to \C$ defined by
\be \str(e_{ab}) = \delta_{ab} (-1)^{\pr a}.\nn\ee
There is a non-degenerate bilinear form on $\gl(m|n)$ defined by taking the supertrace in the defining representation:
\be (E_{ab}, E_{cd}) := \str e_{ab} e_{cd} = \delta_{bc} \delta_{ad} (-1)^{\pr a} .\nn\ee 
This form is invariant and graded-symmetric. That is $0=([Z,X],Y]) + (-1)^{\pr{Z}\pr{X}} (X,[Z,Y])$ and $(X,Y) = (-1)^{\pr X \pr Y} (Y,X)$ for $X,Y,Z\in \gl(m|n)$.

\subsection{Weights and root systems}
We use the standard Cartan decomposition, $\gl(m|n) = \mf n^- \oplus \h \oplus \mf n^+$, in which the Cartan subalgebra $\h\subset\gl(m|n)$ is generated by $E_{aa}$, $1\leq a \leq m+n,$ and the raising (resp. lowering) operators are $E_{ab}\in \mf n^+$ (resp. $E_{ba}\in \mf n^-$), $1\leq a < b \leq m+n$.  

Let $\eps_a$, $1\leq a \leq m+n,$ be the basis of weight space, $\h^*$, dual to $E_{aa}$. 
The restriction of the inner product $(\cdot,\cdot)$ to $\h$ is symmetric, since $E_{aa}$ is even for all $a$. We identify $\h$ and $\h^*$ by means of this inner product. From 
$\eps_a := (-1)^{\pr a} (E_{aa},\cdot)$, we obtain  $\eps_a\equiv (-1)^{\pr a}E_{aa}$. 
Note that $$(\eps_a,\eps_b) = (-1)^{\pr a} \delta_{ab}.$$ 

The roots of $\gl(m|n)$ are $\eps_a-\eps_b\in \h^*$, $1\leq a\neq b\leq m+n$.
We choose as our set of simple positive roots $\alpha_a:=\eps_a - \eps_{a+1}$, $1\leq a\leq m+n-1$. The corresponding simple root vectors are 
\be E_a := E_{a,a+1}, \quad F_a := E_{a+1,a},\nn\ee 
and the simple coroots are  
\be \alpha_a^\vee = H_a := [E_a,F_a] = E_{aa} - (-1)^{\pr a+\pr{a+1}} E_{a+1,a+1},\nn\ee
with $1\leq a<m+n$. These $(E_a,F_a,H_a)$ are the Chevalley generators of the subalgebra $\sl(m|n)\subset \gl(m|n)$. 

The \emph{Cartan matrix} $(a_{ab})_{1\leq a,b< m+n}$ is defined by $[H_a,E_b] = a_{ba} E_b$; thus $a_{ba} = \alpha_b(\alpha^\vee_a) = (\alpha^\vee_a, \alpha_b)$. 
The Cartan matrix depends crucially on the choice of which elements of the basis $e_a\in \C^{m|n}$ are even and which odd. The non-zero entries are those of the on-diagonal $2\times 2$ submatrices,
\be \bmx a_{bb} & a_{b,b+1} \\ a_{b+1, b} & a_{b+1,b+1} \emx, \quad 1\leq b< m+n-2,  \nn\ee
and are given by the following four cases:
\be \begin{matrix} & \pr {E_{b+1}} = 0 & \pr {E_{b+1}} = 1 \\ 
      \pr {E_b} =0  & \bmx 2 & -1 \\  -1 & 2 \emx & \bmx 2 & -1 \\  -1 & 0 \emx_{\phantom \sum} \\
      \pr {E_b} =1  & \bmx 0 & -1 \\  1 & 2 \emx & \bmx 0 & -1 \\  1 & 0 \emx.\nn \end{matrix}\ee

The \emph{symmetrized Cartan matrix} is the symmetric matrix with entries $(\alpha_a,\alpha_b)_{1\leq a,b < m+n}$. It has the following $2\times 2$ block submatrices:
\be\begin{matrix} & \pr{b} = 0 & \pr{b} = 1 \\
\begin{matrix} \pr {b+1} =0 \begin{matrix} \phantom 1\\ \phantom 1\end{matrix}  \\ \pr {b+1} = 1 \begin{matrix} \phantom 1\\\phantom 2 \end{matrix}\end{matrix} &
\overbrace{\begin{matrix} \pr {b+2} = 0 & \pr {b+2} = 1 \\ 
     \bmx 2 & -1 \\  -1 & 2 \emx & \bmx 2 & -1 \\  -1 & 0 \emx_{\phantom \sum} \\
     \bmx 0 & 1 \\  1 & 0 \emx & \bmx 0 & 1 \\  1 & -2 \emx \end{matrix}} & 
\overbrace{\begin{matrix} \pr {b+2} = 0 & \pr {b+2} = 1 \\ 
       \bmx 0 & -1 \\  -1 & 2 \emx & \bmx 0 & -1 \\  -1 & 0 \emx_{\phantom \sum} \\
       \bmx -2 & 1 \\  1 & 0\emx & \bmx -2 & 1 \\  1 & -2 \emx. \end{matrix}}
\end{matrix}
\nn\ee

The \emph{Dynkin diagram} is drawn with a \begin{tikzpicture}[baseline =-3,scale=.6] \filldraw[fill=white] (0,0) circle (2mm);\end{tikzpicture} for each simple root of square length $\pm 2$ and a \begin{tikzpicture}[baseline =-3,scale=.6] \filldraw[fill=white] (0,0) circle (2mm);\draw[thick] (0,0)++(-.15,-.15) -- ++(.3,.3);\draw[thick] (0,0)++(.15,-.15) -- ++(-.3,.3);\end{tikzpicture}  for each simple root of square length $0$. The Dynkin diagram, the Cartan matrix, and the parities of the basis vectors $e_a$ with property (\ref{1}), all encode the same information.
In particular, the \emph{distinguished choice} of Dynkin diagram is the choice that corresponds to the distinguished choice of parities (\ref{dist}).

\begin{exmp}\label{ex1} For the Lie algebra $\gl(4|3)$ three of the possible Dynkin diagrams are 
\be 
\begin{tikzpicture}[baseline =-5,scale=.6] 
\draw[thick] (1,0) -- (6,0);
\foreach \x in {1,2,3,4,5,6}
\filldraw[fill=white] (\x,0) circle (2mm);
\draw[thick] (4,0)++(-.15,-.15) -- ++(.3,.3);\draw[thick] (4,0)++(.15,-.15) -- ++(-.3,.3);
\end{tikzpicture},\qquad
\begin{tikzpicture}[baseline =-5,scale=.6] 
\draw[thick] (1,0) -- (6,0);
\foreach \x in {1,2,3,4,5,6}
\filldraw[fill=white] (\x,0) circle (2mm);
\draw[thick] (5,0)++(-.15,-.15) -- ++(.3,.3);\draw[thick] (5,0)++(.15,-.15) -- ++(-.3,.3);
\draw[thick] (2,0)++(-.15,-.15) -- ++(.3,.3);\draw[thick] (2,0)++(.15,-.15) -- ++(-.3,.3);
\end{tikzpicture},\qquad
\begin{tikzpicture}[baseline =-5,scale=.6] 
\draw[thick] (1,0) -- (6,0);
\foreach \x in {1,2,3,4,5,6}{
\filldraw[fill=white] (\x,0) circle (2mm);
\draw[thick] (\x,0)++(-.15,-.15) -- ++(.3,.3);\draw[thick] (\x,0)++(.15,-.15) -- ++(-.3,.3);}
\end{tikzpicture}.\nn\ee
(The first of these is the distinguished diagram.) The corresponding Cartan matrices are, respectively:
\be \bmx 2 & -1 &   &    &    & 0  \\ 
        -1 & 2  & -1&    &    &  \\
           & -1 & 2 & -1 &    & \\
           &   & -1 & 0  & -1 & \\
           &   &    & 1  & 2  & -1 \\
         0  &   &    &    & -1 & 2\emx ,\quad
\bmx 2 & -1 &    & & &  0\\ 
        -1 & 0  & -1 &  & & \\
           & 1 & 2  & -1 & & \\
           &    & -1 & 2 & -1 & \\
           &    &    & -1 & 0  & -1 \\
        0   &   &    &    &  1 & 2\emx ,\quad
\bmx 0 & -1 &    & & &0 \\ 
        1 & 0  & -1 & & & \\
           & 1 & 0  & -1 & & \\
           &    & 1 & 0 & -1 & \\
           &    &    & 1 & 0  & -1 \\
       0    &   &    &    & 1 & 0\emx. \nn\ee 
\end{exmp} 

The \emph{Weyl group} $W$ of $\gl(m|n)$ is by definition the Weyl group $S_m\times S_n$ of the even subalgebra $\gl(m)\oplus \gl(n)$. It is realized as the group of all permutations $\sigma$ of the basis vectors $e_a$ that preserve parity, $\pr{\sigma(a)}=\pr a$.
Thus, as is well known, not every pair of Borel subalgebras of $\gl(m|n)$ are Weyl-conjugate, in contrast to usual situation for $\gl(m)$. 

\subsection{Finite-dimensional irreducible modules}\label{fdmod}
Given $\lambda\in \h^*$, let $L(\lambda)$ denote the irreducible $\gl(m|n)$-module with highest weight $\lambda$. 
That is, $L(\lambda)$ is the irreducible quotient of the Verma module $\ind^{\gl(m|n)}_{\h\oplus \mf n^+} \C v_\lambda$ containing $v_\lambda$, where $v_\lambda\neq 0$ is such that $hv_\lambda = \lambda(h) v_\lambda$ for all $h\in \h$ and $xv_\lambda=0$ for all $x\in \mf n^+$. (We assume $v_\lambda$ is homogeneous with respect to the $\Z_2$-gradation, but we leave its parity unspecified.)

For the distinguished choice of Dynkin diagram (only), given by \eqref{dist}, we have the following. The finite-dimensional irreducible representations of $\gl(m|n)$ are, up to isomorphism, precisely those $L(\lambda)$ for which $\lambda=\sum_{i=1}^m\lambda_i\eps_i+\sum_{j=1}^{n} \mu_j \eps_{m+j}$ is such that $\lambda_i-\lambda_{i+1}\in\Z_{\geq 0}$ and $\mu_j-\mu_{j+1}\in \Z_{\geq 0}$ for all $1\leq i\leq m-1$ and $1\leq j\leq n-1$.

\subsection{Polynomial modules}\label{polynomialmodules}
The category of finite-dimensional modules of $\gl(m|n)$ is not in general semisimple, but the smaller category of finite-dimensional \emph{polynomial modules} is. By definition, a $\gl(m|n)$-module $V$ is polynomial if $\h$ acts semisimply on $V$ and $(\lambda,\eps_i)\in \Z_{\geq 0}$ for every weight $\lambda$ of $V$. The irreducible polynomial modules are precisely those irreducibles that occur as submodules of tensor powers of the defining representation $\C^{m|n}$. 

\subsection{Hook diagrams}
Irreducible polynomial modules are in 1-1 correspondence with \emph{hook partitions}. A hook partition for $\gl(m|n)$ is a partition $\mu= (\mu_1, \mu_2,\dots)$, $\mu_1\geq \mu_2\geq\dots$, such that $\mu_{m+1}\leq n$. Hook partitions correspond to \emph{hook diagrams}: that is, to Young diagrams with no box in position $(m+1,n+1)$. We denote by $V(\mu)$ the irreducible polynomial module associated to a hook partition $\mu$. The highest weight of $V(\mu)$ depends on the choice of parities $\pr a$, $1\leq a \leq m+n$, i.e. on the choice of Dynkin diagram, as follows:
\be V(\mu) \cong L(\lambda) \quad\text{with}\quad \lambda=\sum_{a:\pr a = 0} \eps_a \left(\mu_a - \sum_{b<a} \pr b\right) 
 +  \sum_{a:\pr a = 1} \eps_a \left(\mu'_a - \sum_{b<a} (1-\pr b)\right),
\label{VLeqn}\ee
where $\mu'$ is the conjugate partition. This expression has a natural pictorial interpretation, as the following example illustrates.

\begin{exmp}\label{diagexmp} Consider $\gl(4|3)$ and the hook partition $\mu = (7,6,4,4,3,3,1,1,1)$.
Then with respect to each of the three Dynkin diagrams given in Example \ref{ex1} in turn, the components in the basis $\eps_a$ of the highest weight of the polynomial module $V(\mu)$ are:
\be (7,6,4,4,5,2,2),\quad (7,6,7,4,4,1,1),\quad (7,8,5,4,2,3,1). \nn\ee
These can be read off from the hook diagram as follows:
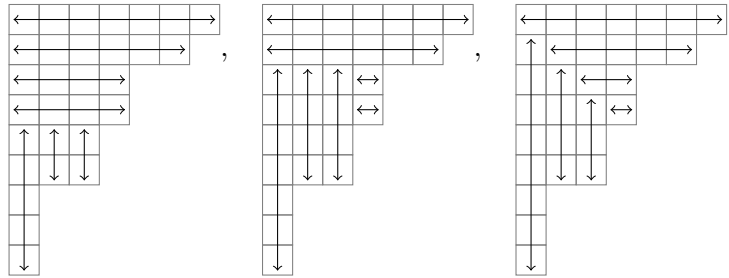
\be \begin{tikzpicture}[baseline=-25pt,shorten >=-4pt,shorten <= - 4pt,scale=.4,yscale=-1,every node/.style={minimum size=.4cm,inner sep=0mm,draw,gray,rectangle}]
\foreach \x in {1,2,3,4,5,6,7} {\node at (\x,1) {};}
\foreach \x in {1,2,3,4,5,6} {\node at (\x,2) {};}
\foreach \x in {1,2,3,4} {\node at (\x,3) {};}
\foreach \x in {1,2,3,4} {\node at (\x,4) {};}
\foreach \x in {1,2,3} {\node at (\x,5) {};}
\foreach \x in {1,2,3} {\node at (\x,6) {};}
\foreach \x in {1} {\node at (\x,7) {};}
\foreach \x in {1} {\node at (\x,8) {};}
\foreach \x in {1} {\node at (\x,9) {};}
\draw[<->] (1,1) -- (7,1) ;
\draw[<->] (1,2) -- (6,2) ;
\draw[<->] (1,3) -- (4,3) ;
\draw[<->] (1,4) -- (4,4) ;
\draw[<->] (1,5) -- (1,9) ;
\draw[<->] (2,5) -- (2,6) ;
\draw[<->] (3,5) -- (3,6) ;
\end{tikzpicture},\quad
\begin{tikzpicture}[baseline=-25pt,shorten >=-4pt,shorten <= - 4pt,
scale=.4,yscale=-1,every node/.style={minimum size=.4cm,inner sep=0mm,draw,gray,rectangle}]
\foreach \x in {1,2,3,4,5,6,7} {\node at (\x,1) {};}
\foreach \x in {1,2,3,4,5,6} {\node at (\x,2) {};}
\foreach \x in {1,2,3,4} {\node at (\x,3) {};}
\foreach \x in {1,2,3,4} {\node at (\x,4) {};}
\foreach \x in {1,2,3} {\node at (\x,5) {};}
\foreach \x in {1,2,3} {\node at (\x,6) {};}
\foreach \x in {1} {\node at (\x,7) {};}
\foreach \x in {1} {\node at (\x,8) {};}
\foreach \x in {1} {\node at (\x,9) {};}
\draw[<->] (1,1) -- (7,1) ;
\draw[<->] (1,2) -- (6,2) ;
\draw[<->] (1,3) -- (1,9) ;
\draw[<->] (2,3) -- (2,6) ;
\draw[<->] (3,3) -- (3,6) ;
\draw[<->] (3.99,3) -- (4.01,3);
\draw[<->] (3.99,4) -- (4.01,4);
\end{tikzpicture},\quad
\begin{tikzpicture}[baseline=-25pt,shorten >=-4pt,shorten <= - 4pt,scale=.4,yscale=-1,every node/.style={minimum size=.4cm,inner sep=0mm,draw,gray,rectangle}]
\foreach \x in {1,2,3,4,5,6,7} {\node at (\x,1) {};}
\foreach \x in {1,2,3,4,5,6} {\node at (\x,2) {};}
\foreach \x in {1,2,3,4} {\node at (\x,3) {};}
\foreach \x in {1,2,3,4} {\node at (\x,4) {};}
\foreach \x in {1,2,3} {\node at (\x,5) {};}
\foreach \x in {1,2,3} {\node at (\x,6) {};}
\foreach \x in {1} {\node at (\x,7) {};}
\foreach \x in {1} {\node at (\x,8) {};}
\foreach \x in {1} {\node at (\x,9) {};}
\draw[<->] (1,1) -- (7,1) ;
\draw[<->] (1,2) -- (1,9) ;
\draw[<->] (2,2) -- (6,2) ;
\draw[<->] (2,3) -- (2,6) ;
\draw[<->] (3,3) -- (4,3) ;
\draw[<->] (3,4) -- (3,6) ;
\draw[<->] (3.99,4) -- (4.01,4);
\end{tikzpicture}. \nn\ee
\end{exmp}

\subsection{Pieri rule}\label{pieri}
The defining representation $\C^{m|n}$ is the polynomial module corresponding to a diagram with a single box. 
Just as in usual case of $\gl(m)$, one has the following Pieri rule for the decomposition into irreducibles of the tensor product of $\C^{m|n}$ with any polynomial module:
\be V(\mu) \otimes \C^{m|n} = \bigoplus_{\rho} V(\rho) \ee
where the direct sum is over all those hook partitions $\rho$ whose hook diagrams can be obtained by adding one box to the hook diagram of $\mu$. 

We shall need the following.
\begin{lem}\label{boxlem}
Let $\mu$ be a hook partition and $\lambda=\sum_{a=1}^{m+n}\eps_a\lambda_a\in \h^*$ the corresponding weight, i.e.  $L(\lambda) \cong V(\mu)$, cf. \eqref{VLeqn}. Then 
\be \lambda_r - (-1)^{\pr{a+1}+\pr r} \lambda_{a+1} + (-1)^{\pr r}\sum_{b=r+1}^{a} (-1)^{\pr b} \geq  0  \label{bd}\ee
for all $r$ and $a$ with $1\leq r\leq a<m+n$. 

Moreover, for each $a$, $1\leq a<m+n$, the weight $\lambda+\eps_{a+1}$ corresponds to a hook partition if and only if the inequality \eqref{bd} is strict for every  $r$ with $1\leq r\leq a$.
\end{lem}
\begin{proof}
Consider those pairs $(r,a+1)$, $1\leq r\leq a< m+n$, such that $\pr r = \pr {a+1} =0$. If $\pr b = 1$ for all $b$ with $r+1\leq b \leq a$ then the relevant part of the diagram of $\mu$ resembles
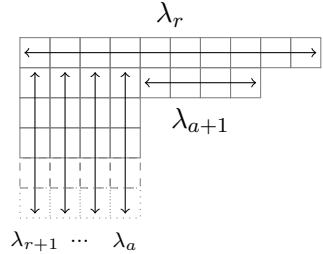
\be \begin{tikzpicture}[baseline=-25pt,shorten >=-4pt,shorten <= - 4pt,scale=.4,yscale=-1,every node/.style={minimum size=.4cm,inner sep=0mm,draw,gray,rectangle}]
\foreach \x in {1,2,3,4,5,6,7,8,9,10} {\node at (\x,1) {};}
\foreach \x in {1,2,3,4,5,6,7,8} {\node at (\x,2) {};}
\foreach \x in {1,2,3,4} {\node at (\x,3) {};}
\foreach \x in {1,2,3,4} {\node at (\x,4) {};}
\foreach \x in {1,2,3,4} {\node[dashed] at (\x,5) {};}
\foreach \x in {1,2,3,4} {\node[dotted] at (\x,6) {};}
\draw[<->] (1,1) -- (10,1) node[above=3mm,midway,draw=none,black] {$\lambda_r$};
\draw[<->] (5,2) -- (8,2)  node[below=3mm,midway,draw=none,black] {$\lambda_{a+1}$};
\draw[<->] (1,2) -- (1,6)  node[below=3mm,draw=none,black] {$\scriptstyle{\lambda_{r+1}}$};
\draw[<->] (2,2) -- (2,6)  node[below=3mm,draw=none,black] {$\scriptstyle{\quad\dots}$};
\draw[<->] (3,2) -- (3,6) ;
\draw[<->] (4,2) -- (4,6)  node[below=3mm,draw=none,black] {$\scriptstyle{\lambda_{a}}$};
\end{tikzpicture} \nn\ee
and one sees that necessarily $\lambda_{r} - \lambda_{a+1} - \sum_{b=r+1}^a1 \geq 0$.   
To treat the general case with $\pr r=\pr{a+1}=0$,  one adds together such inequalities, and finds by an induction that 
\be \lambda_r - \lambda_{a+1} + \sum_{b=r+1}^a (-1)^{\pr b} \geq 0.\label{00}\ee 
(Actually, one finds the stronger result that the left-hand side is greater than or equal to the number of distinct $b$, $r+1\leq b\leq a$, for which $\pr b = 0$. But we do not need this.) 


Similarly, one finds that 
\be -\lambda_r + \lambda_{a+1} + \sum_{b=r+1}^a (-1)^{\pr b} \leq 0\label{11}\ee
for all pairs $(r,a+1)$, $1\leq r\leq a<m+n$, such that $\pr r=\pr{a+1} = 1$. 

Next, consider pairs $(r,a+1)$, $1\leq r\leq a<m+n$, such that $\pr r =0$ and $\pr {a+1} =1$. If $\pr b=1$ for all $r+1\leq b\leq a$ then the relevant part of the diagram resembles
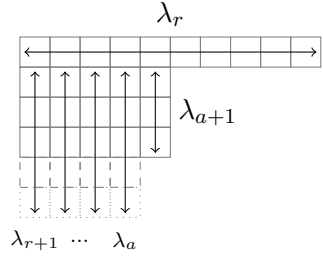
\be\nn \begin{tikzpicture}[baseline=-25pt,shorten >=-4pt,shorten <= - 4pt,scale=.4,yscale=-1,every node/.style={minimum size=.4cm,inner sep=0mm,draw,gray,rectangle}]
\foreach \x in {1,2,3,4,5,6,7,8,9,10} {\node at (\x,1) {};}
\foreach \x in {1,2,3,4,5} {\node at (\x,2) {};}
\foreach \x in {1,2,3,4,5} {\node at (\x,3) {};}
\foreach \x in {1,2,3,4,5} {\node at (\x,4) {};}
\foreach \x in {1,2,3,4} {\node[dashed] at (\x,5) {};}
\foreach \x in {1,2,3,4} {\node[dotted] at (\x,6) {};}
\draw[<->] (1,1) -- (10,1) node[above=3mm,midway,draw=none,black] {$\lambda_r$};
\draw[<->] (5,2) -- (5,4)  node[right=3mm,midway,draw=none,black] {$\lambda_{a+1}$};
\draw[<->] (1,2) -- (1,6)  node[below=3mm,draw=none,black] {$\scriptstyle{\lambda_{r+1}}$};
\draw[<->] (2,2) -- (2,6)  node[below=3mm,draw=none,black] {$\scriptstyle{\quad\dots}$};
\draw[<->] (3,2) -- (3,6) ;
\draw[<->] (4,2) -- (4,6)  node[below=3mm,draw=none,black] {$\scriptstyle{\lambda_{a}}$};
\end{tikzpicture} \ee
and one sees that $\lambda_r - \sum_{b=r+1}^a 1 \geq 0$, with equality possible only if $\lambda_{a+1}=0$. By adding inequalities of the type \eqref{00} one can then treat the general case with $\pr r=0$, $\pr{a+1} =1$. 
One finds that $\lambda_r + \sum_{b=r+1}^a (-1)^{\pr b} \geq 0$ with equality possible only if $\lambda_{a+1}=0$. 
An equivalent statement is that 
\be \lambda_r +\lambda_{a+1} + \sum_{b=r+1}^a (-1)^{\pr b} \geq 0,\nn\ee with equality possible only if $\lambda_{a+1}=0$.

Finally, for all pairs $(r,a+1)$, $r\leq a$, such that $\pr r = 1$, $\pr{a+1} = 0$ we find similarly that $-\lambda_r - \lambda_{a+1} + \sum_{b=r+1}^a(-1)^{\pr b} \leq 0$ with equality possible only if $\lambda_{a+1} = 0$.

The ``moreover'' part follows.
\end{proof}

\subsection{Quadratic Casimir}\label{casimir}
The Casimir element $\mc C := \sum_{a,b=1}^{m+n}E_{ab}E_{ba}(-1)^{\pr b}\in U(\gl(m|n))$ is central. Its value on the irreducible highest weight representation $L(\lambda)$ is
\be  \mc C_\lambda := \sum_{a=1}^{m+n} \left(\lambda_a^2 (-1)^{\pr a} + \lambda_a \left( \sum_{b>a} (-1)^{\pr a} - \sum_{b<a} (-1)^{\pr b} \right) \right).\nn\ee 
This can be computed by evaluating $\mc C$ on the highest weight vector $v_\lambda$. Furthermore, by direct calculation, we have the following.
\begin{lem}\label{clem} Let $\lambda= \sum_{a=1}^{m+n}\lambda_a\eps_a$. For all $r,a$ such that $1\leq r\leq a\leq m+n-1$ we have
\be-\half\left( \mc C_{\lambda+\eps_{a+1}} - \mc C_{\lambda + \eps_r }\right)
 = (-1)^{\pr r} \lambda_r - (-1)^{\pr{a+1}} \lambda_{a+1} + \sum_{b=r}^a (-1)^{\pr b}.\nn\ee
\qed
\end{lem}

\subsection{Shapovalov form}
The linear map $\varphi: \gl(m|n) \to \gl(m|n)$ defined by
\be \varphi: E_{ab} \mapsto E_{ba} \ee
is an anti-involution (i.e. $\varphi^2 = \id$ and $\varphi([X,Y]) = [\varphi(Y),\varphi(X)]$) with the property that $\varphi(\h)=\h$ and $\varphi(\mf n^\pm) = \mf n^\mp$. On every irreducible module $L(\lambda)$, one can define a symmetric non-degenerate bilinear form, $\left<\cdot,\cdot\right>$, the \emph{Shapovalov form}. Let $v_\lambda$ be a highest weight vector. Then $\left<\cdot,\cdot\right>$ is uniquely defined by 
\be \left< v_\lambda,v_\lambda\right> = 1,\nn\ee 
and
\be \left< E_{ab} v,w\right> = \left< v,\varphi(E_{ab}) w\right> = \left< v, E_{ba}w\right>. \label{Sflip}\ee
(See \cite{Gorelik}.)
The Shapovalov form is even i.e. $\left<v,w\right>=0$ whenever $v,w$ are homogeneous elements of different parities. Moreover $\left<v,w\right>=0$ if $v$ and $w$ are weight vectors of distinct weights.

The form extends to a non-degenerate form, the \emph{tensor Shapovalov form}, on tensor products of such irreducibles, according to 
\be \left< v_1\otimes\dots\otimes v_N, w_1\otimes\dots \otimes w_N\right> = \left<v_1,w_1\right>\dots \left<v_N,w_N\right>. \nn\ee
(Note the lack of signs in our conventions here.) 

The tensor Shapovalov form obeys \eqref{Sflip}.
Let $V(\mu)$ be an irreducible polynomial module. If $V(\mu)$ is a submodule of $(\C^{m|n})^{\otimes N}$ then the tensor Shapovalov form restricted to $V(\mu)$ coincides with the Shapovalov form up to an overall  non-zero  multiplicative constant.

The basis $e_a$ of the defining representation $\C^{m|n}$ is orthonormal with respect to the Shapovalov form. It follows that the sesquilinear form corresponding to the Shapovalov form is positive definite on $\C^{m|n}$. 
Hence the sesquilinear form corresponding to the tensor Shapovalov form is positive definite on tensor products of $\C^{m|n}$ and submodules thereof. 
Therefore the sesquilinear form corresponding to the Shapovalov form is Hermitian on all irreducible polynomial representations. 

\section{Gaudin Hamiltonians}\label{Ham sec}

Let $\bm\lambda = (\lambda^{(i)})_{i=1}^N$ be an $N$-tuple, $N>1$, of elements $\lambda^{(i)}\in \h^*$ such that the irreducible $\gl(m|n)$-modules $L(\lambda^{(i)})$ are all finite-dimensional. Let $\bm z = (z_i)_{i=1}^N$ be an $N$-tuple of pairwise distinct points in $\C$. Given this data, the \emph{(quadratic) Gaudin hamiltonians} are the linear maps $\HH_i\in \End\left(L(\lambda_1) \otimes \dots L(\lambda_N)\right)$
given by
\be\label{Gaudin} \HH_i := \sum_{\substack{j=1 \\ j\neq i}}^N  \frac{\sum_{a,b} E_{ab}^{(i)} E_{ba}^{(j)} (-1)^{\pr b}}{z_i-z_j}, \qquad 1\leq i\leq N,\ee
where $E_{ab}^{(k)}= \underset{k-1}{\underbrace{ 1\otimes \dots \otimes 1}}\otimes E_{ab} \otimes \underset{N-k}{\underbrace{ 1\otimes \dots \otimes 1}}$ for $1\leq k\leq N$.

The following is a standard extension of the usual $\gl(m)$ result. We sketch the proof for the reader's convenience.
\begin{prop} The Gaudin Hamiltonians $\HH_i$:\label{Hcomprop}
\begin{enumerate}
\item are mutually commuting: $[\HH_i,\HH_j] = 0$ for all $i,j$;
\item commute with the action of $\gl(m|n)$: $[\HH_i, X]=0$ for all $i$ and all $X\in \gl(m|n)$; 
\item are symmetric operators with respect to the tensor Shapovalov form: $\left< \HH_i v,w\right> = \left< v,\HH_i w \right>$.
\item sum to zero: $\sum_{i=1}^{N}\HH_i = 0$. 
\end{enumerate}
\end{prop}
\begin{proof} Part (1) is verified by direct calculation, making use of the identity
\be \frac 1{z_i-z_j} \frac 1{z_j-z_k} + \frac 1{z_j-z_k} \frac 1{z_k-z_i} + \frac 1{z_k-z_i} \frac 1{z_i-z_j} =0.\nn\ee
For part (2), we recall the quadratic Casimir $\mc C$ from \S\ref{casimir}, and observe that 
\be \sum_{a,b} E_{ab}\otimes E_{ba}(-1)^{\pr b} = \Delta \mc C - \mc C \otimes 1 - 1\otimes \mc C\label{Ceqn}.\ee
For part (3), one has
\begin{align*} &\left< \left(E_{ab}\otimes E_{ba} (-1)^{\pr b}\right)(v_1\otimes v_2), w_1\otimes w_2 \right>\\
& = \left< E_{ab} v_1 \otimes E_{ba}v_2  (-1)^{\pr b+ \pr{v_1}(\pr a+ \pr b)}, w_1\otimes w_2 \right>\\
&= \left< E_{ab} v_1,w_1\right> \left< E_{ba} v_2,w_2 \right>  (-1)^{\pr b+ \pr{v_1}(\pr a+ \pr b)} \\
&= \left< v_1,E_{ba} w_1\right> \left< v_2,E_{ab} w_2 \right>  (-1)^{\pr b+ \pr{v_1}(\pr a+ \pr b)} \\
& = \left< v_1 \otimes v_2  , (E_{ba}\otimes E_{ab})(w_1\otimes w_2) \right> (-1)^{\pr b+ (\pr{v_1}+\pr{w_1})(\pr a+ \pr b)}\\
& = \left< v_1 \otimes v_2  , \left(E_{ba}\otimes E_{ab}(-1)^{\pr a}\right)(w_1\otimes w_2) \right> (-1)^{(1+\pr{v_1}+\pr{w_1})(\pr a+ \pr b)}.
\end{align*}
Since the Shapovalov form is even, here $\left< v_1,E_{ba} w_1\right>$ can be nonzero only if $E_{ab}$ is even \emph{or} the parities of $v_1$ and $w_1$ are distinct: that is, only if $(-1)^{(1+\pr{v_1}+\pr{w_1})(\pr a+ \pr b)}=1$. Therefore it follows from the above that $\left< \left(E_{ab}\otimes E_{ba} (-1)^{\pr b}\right)(v_1\otimes v_2), w_1\otimes w_2 \right> =  \left< v_1 \otimes v_2  , \left(E_{ba}\otimes E_{ab}(-1)^{\pr a}\right)(w_1\otimes w_2) \right>$, and hence the result. 

Part (4) follows from the fact that $\sum_{a,b} E_{ab}^{(i)} E_{ba}^{(j)} (-1)^{\pr b}$ is symmetric in its tensor factors, i.e. $$\sum_{a,b} E_{ab}^{(i)} E_{ba}^{(j)} (-1)^{\pr b} = \sum_{a,b} E_{ba}^{(j)} E_{ab}^{(i)}  (-1)^{\pr b}(-1)^{(\pr a + \pr b)^2} = \sum_{a,b} E_{ba}^{(j)} E_{ab}^{(i)}  (-1)^{\pr a} =\sum_{a,b} E_{ab}^{(j)} E_{ba}^{(i)} (-1)^{\pr b}.$$
\end{proof}

In particular, notice that if all $z_i$ are real numbers and all modules $L(\lambda^{(i)})$ are polynomial then the Gaudin Hamiltonians are Hermitian operators with respect to the Hermitian form corresponding to the Shapovalov form. In particular, the Gaudin Hamiltonians are diagonalizable in this case.

\section{Bethe ansatz}\label{BA sec}

\subsection{Bethe equations}
As in the previous section, we fix a collection $\bm{\lambda} = (\lambda^{(i)})_{i = 1}^N$ of weights $\lambda^{(i)} \in \h^{\ast}$ together with a collection $\bm{z} = (z_i)_{i = 1}^N$ of pairwise distinct points $z_i \in \C$. Let now 
$\bm{l} = (l_i)_{i = 1}^{n+m-1}$  be a collection of $n+m-1$ non-negative integers. Set
\begin{align*} 
\lambda := \sum_{s=1}^N \lambda^{(s)} ,\qquad 
l := \sum_{p=1}^{m+n-1} l_p ,\qquad 
\alpha(\bm{l}) := \sum_{p=1}^{m+n-1} l_p \alpha_p, \qquad \lambda^{\infty} := \lambda - \alpha(\bm{l}).
\end{align*}
Let $c$ be any function from $\{ 1, \ldots, l \}$ to $\{ 1, \ldots, n+m-1 \}$ such that the set $c^{-1}(i)$ has $l_i$ elements, for all $1\leq i\leq  n+m-1$. For example, one can make the choice
\be \big(c(1),c(2),\dots,c(l)\big) = \big(\underset{l_1}{\underbrace{1,\dots,1}},
\underset{l_2}{\underbrace{2,\dots,2}},\dots,
\underset{l_{n+m-1}}{\underbrace{n+m-1,\dots,n+m-1}}\big)\label{lidef}.\ee  

The \emph{Bethe equations} are the following system of algebraic equations on a collection $\bm{t} = (t_i)_{i = 1}^l$ of variables $t_i \in \C$:
\begin{equation} \label{BAE}
- \sum_{s = 1}^N \frac{\big(\alpha_{c(i)}, \lambda^{(s)} \big)}{t_i - z_s} + \sum_{\substack{j=1\\ j \neq i}}^{l} \frac{\big(\alpha_{c(i)}, \alpha_{c(j)}\big)}{t_i - t_j} = 0, \qquad i = 1, \ldots, l.
\end{equation}
We say $c(i) \in \{1,\dots,n+m-1\}$ is the \emph{colour} of the Bethe variable $t_i$.

\subsection{Weight function}
For each $i$, fix a highest weight vector $v_{\lambda^{(i)}}$ in the highest weight irreducible $\gl(m|n)$-module $L(\lambda^{(i)})$. Consider the tensor product $L(\bm{\lambda}) = L(\lambda^{(1)}) \otimes \ldots \otimes L(\lambda^{(N)})$. Given $\mu \in \h^{\ast}$ we denote by $L(\bm{\lambda})_{\mu}$ the weight space of weight $\mu$. The \emph{weight function} is a vector $w(\bs z,\bs t)$ in $L(\bm{\lambda})_{\lambda^{\infty}}$ depending on parameters $\bs z=(z_1,\dots,z_N)$ and variables $\bs t=(t_1,\dots,t_l)$ constructed as follows.

Let an \emph{ordered partition} of $\{1,\dots,l\}$ into $N$ parts be a composition  $p_1+p_2+\dots + p_N=l$, $(p_1,\dots,p_N)\in \Z_{\geq 0}^N$, of $l$ into $N$ parts, together with a tuple $\bm n = (n^1_1, \ldots, n^1_{p_1}; \ldots ; n^N_1, \ldots, n^N_{p_N})$ which is a permutation of $(1,2,\dots,l)$.  Let $P_{l, N}$ be the set of all such ordered partitions. We call $\bm p=(p_1,\dots,p_N)$ the \emph{shape} of the ordered partition, and will often leave it implicit.

To every ordered partition ${\bm n} \in P_{l, N}$, associate a vector $F_{\bm n} v$ in $L(\bm{\lambda})_{\lambda^{\infty}}$ and a rational function $\omega_{\bm n}$ of $\bm{z}$ and $\bm{t}$ defined as
\begin{align*}
F_{\bm n} v &:= F_{c(n^1_1)} \ldots F_{c(n^1_{p_1})} v_{\lambda^{(1)}} \otimes \ldots \otimes F_{c(n^N_1)} \ldots F_{c(n^N_{p_N})} v_{\lambda^{(N)}},\\
\omega_{\bm n}(\bm{z}, \bm{t}) &:= \omega_{n^1_1, \ldots, n^1_{p_1}}(z_1, \bm{t}) \ldots \omega_{n^N_1, \ldots, n^N_{p_N}}(z_N, \bm{t})
\end{align*}
where for any $\{ n_1, \ldots, n_j \} \subset \{ 1, \ldots, l \}$, all distinct, we write
\begin{equation*}
\omega_{n_1, \ldots, n_j}(z_s, \bm{t}) := \frac{1}{(t_{n_1} - t_{n_2}) \ldots (t_{n_{j-1}} - t_{n_j}) (t_{n_j} - z_s)}.
\end{equation*}
Furthermore, there is also a sign associated to each ${\bm n} \in P_{l, N}$, denoted $(-1)^{|{\bm n}|}$, defined as follows. Referring to a transposition $(i j) \in S_l$ as being odd if and only if $|F_{c(i)}|=|F_{c(j)}|=1$, then $|{\bm n}|$ counts modulo 2 the total number of odd transpositions in the permutation $\sigma$ which sends $(1, \ldots, l)$ to $(n^1_1, \ldots, n^1_{p_1}, \ldots, n^N_1, \ldots, n^N_{p_N})$. In terms of $\sigma$, this sign may be expressed as
\begin{equation*}
(-1)^{|{\bm n}|} := \prod_{i = 1}^l \prod_{\substack{j > i\\ \sigma(j) < \sigma(i)}} (-1)^{|F_{c(i)}| |F_{c(j)}|}.
\end{equation*}

In terms of these notations, the weight function is defined as
\begin{equation} \label{weight function}
w(\bm{z}, \bm{t}) := \sum_{{\bm n} \in P_{l, N}} m_{\bm n},\quad m_{\bm n} := (-1)^{|{\bm n}|} \omega_{\bm n}(\bm{z}, \bm{t}) F_{\bm n} v.
\end{equation}

If a permutation $\bm n_1\in S_l$ permutes variables of the same parity, then we have 
$(-1)^{|{\bm n}_1{\bm n}_2|}=(-1)^{|{\bm n}_1|}(-1)^{|{\bm n}_2|}$. It follows that
the weight function is symmetric (resp. skew-symmetric) under the interchange of any pair of variables $t_i$, $t_j$ such that $c(i) = c(j)$ with $\pr{F_{c(i)}} =0$ (resp. $\pr{F_{c(i)}} = 1$). 
\begin{exmp}\label{exmpgl21}
Consider $\gl(2|1)$ with the choice of Dynkin diagram \begin{tikzpicture}[baseline =-3,scale=.6] 
\draw[thick] (1,0) -- (2,0);
\foreach \x in {1,2}
\filldraw[fill=white] (\x,0) circle (2mm);
\draw[thick] (2,0)++(-.15,-.15) -- ++(.3,.3);\draw[thick] (2,0)++(.15,-.15) -- ++(-.3,.3);
\draw[thick] (1,0)++(-.15,-.15) -- ++(.3,.3);\draw[thick] (1,0)++(.15,-.15) -- ++(-.3,.3);
\end{tikzpicture}, i.e. both $F_1=E_{21}$ and $F_2=E_{32}$ are odd. If, for example, $\bm{l} = (1,1)$ then
\begin{align*}
w(\bm{z}, \bm{t}) &= \frac{F_1 F_2 v_{\lambda^{(1)}} \otimes v_{\lambda^{(2)}}}{(t_1 - t_2)(t_2 - z_1)}
- \frac{F_2 F_1 v_{\lambda^{(1)}} \otimes v_{\lambda^{(2)}}}{(t_2 - t_1)(t_1 - z_1)}
+ \frac{F_1 v_{\lambda^{(1)}} \otimes F_2 v_{\lambda^{(2)}}}{(t_1 - z_1)(t_2 - z_2)}\\
&- \frac{F_2 v_{\lambda^{(1)}} \otimes F_1 v_{\lambda^{(2)}}}{(t_2 - z_1)(t_1 - z_2)}
+ \frac{v_{\lambda^{(1)}} \otimes F_1 F_2 v_{\lambda^{(2)}}}{(t_1 - t_2)(t_2 - z_2)}
- \frac{v_{\lambda^{(1)}} \otimes F_2 F_1 v_{\lambda^{(2)}}}{(t_2 - t_1)(t_1 - z_2)},
\end{align*}
while if $\bm{l} = (0, 2)$ we find (note that $F_2^2=0$ identically)
\begin{align*}
w(\bm{z}, \bm{t}) &= \frac{F_2 v_{\lambda^{(1)}} \otimes F_2 v_{\lambda^{(2)}}}{(t_1 - z_1)(t_2 - z_2)} - \frac{F_2 v_{\lambda^{(1)}} \otimes F_2 v_{\lambda^{(2)}}}{(t_2 - z_1)(t_1 - z_2)}\\
 &=  -\frac{(t_1-t_2)(z_1-z_2)}{(t_1-z_1)(t_1-z_2)(t_2-z_1)(t_2-z_2)} F_2 v_{\lambda^{(1)}} \otimes F_2 v_{\lambda^{(2)}}
.
\end{align*}
\end{exmp}

Each term in the sum \eqref{weight function} may be conveniently visualized diagrammatically as follows. We mark the set of all points $z_i \in \C$ and $t_j \in \C$ on which the weight function depends using crosses and dots respectively. There is then a 1-1 correspondence between each term in the sum \eqref{weight function} and every possible graph consisting of $N$ chains each emanating from one of the $z_i$'s and such that each $t_j$ belongs to one chain exactly.
\begin{exmp}\label{graphexmp}
Consider the case $N = 3$ and $l = 7$. We sketch the location of each point $(z_i)_{i=1}^3$ using crosses and those of the points $(t_j)_{j=1}^7$ using dots.
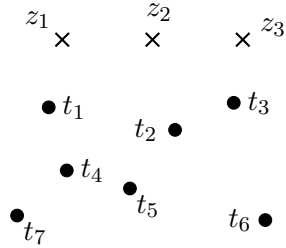
\begin{equation*}
\begin{tikzpicture}[baseline =-5,scale=.6]
\draw[thick, black] (-.15, -.15) -- (.15, .15)
			    (-.15, .15) -- (.15, -.15) node[above left=1mm] {$z_1$};
\filldraw [black] (-.3,-1.5) circle (4pt) node[right=.5mm] {$t_1$}
		      (.1,-2.9) circle (4pt) node[right=.5mm] {$t_4$}
      		      (-1,-3.9) circle (4pt) node[below right=-.5mm] {$t_7$};

\draw[thick, black] ( 2 - .15, - .15) -- ( 2 + .15, .15)
			    ( 2 - .15, .15) -- ( 2 + .15, -.15) node[above=2mm] {$z_2$};
\filldraw [black] (2.5,-2) circle (4pt) node[left=1mm] {$t_2$}
		      (1.5,-3.3) circle (4pt) node[below right=-.5mm] {$t_5$};

\draw[thick, black] ( 4 - .15, - .15) -- ( 4 + .15, .15)
			    ( 4 - .15, .15) -- ( 4 + .15, -.15) node[above right] {$z_3$};
\filldraw [black] (3.8,-1.4) circle (4pt) node[right=.5mm] {$t_3$}
		      (4.5,-4) circle (4pt) node[left=.5mm] {$t_6$};
\end{tikzpicture}
\end{equation*}
Let us suppose for concreteness that the odd lowering operators correspond to the Bethe roots $t_1$, $t_5$ and $t_6$. That is, $|F_{c(i)}| = 1$ if $i = 1, 5, 6$ and $|F_{c(i)}| = 0$ otherwise. Each way of joining all the $t_j$'s along chains to one of the $z_i$'s then corresponds uniquely to a term in the sum \eqref{weight function}. Two possible such terms and their corresponding graphs are
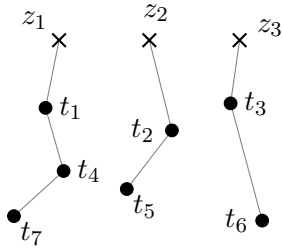
\begin{equation*}
\begin{tabular}{ccc}
\begin{tikzpicture}[baseline =-5,scale=.6]
\draw[thin, gray] (0,0) -- (-.3,-1.5) -- (.1, -2.9) -- (-1,-3.9);
\draw[thick, black] (-.15, -.15) -- (.15, .15)
			    (-.15, .15) -- (.15, -.15) node[above left=1mm] {$z_1$};
\filldraw [black] (-.3,-1.5) circle (4pt) node[right=.5mm] {$t_1$};
\filldraw [black] (.1,-2.9) circle (4pt) node[right=.5mm] {$t_4$}
      		      (-1,-3.9) circle (4pt) node[below right=-.5mm] {$t_7$};

\draw[thin, gray] (2,0) -- (2.5,-2) -- (1.5, -3.3);
\draw[thick, black] ( 2 - .15, - .15) -- ( 2 + .15, .15)
			    ( 2 - .15, .15) -- ( 2 + .15, -.15) node[above=2mm] {$z_2$};
\filldraw [black] (2.5,-2) circle (4pt) node[left=1mm] {$t_2$};
\filldraw [black] (1.5,-3.3) circle (4pt) node[below right=-.5mm] {$t_5$};

\draw[thin, gray] (4,0) -- (3.8,-1.4) -- (4.5, -4);
\draw[thick, black] ( 4 - .15, - .15) -- ( 4 + .15, .15)
			    ( 4 - .15, .15) -- ( 4 + .15, -.15) node[above right] {$z_3$};
\filldraw [black] (3.8,-1.4) circle (4pt) node[right=.5mm] {$t_3$}
		      (4.5,-4) circle (4pt) node[left=.5mm] {$t_6$};
\end{tikzpicture}
& \qquad\qquad & \raisebox{-9mm}{$\displaystyle \frac{F_{c(7)} F_{c(4)} F_{c(1)} v_{\lambda^{(1)}} \otimes F_{c(5)} F_{c(2)} v_{\lambda^{(2)}} \otimes F_{c(6)} F_{c(3)} v_{\lambda^{(3)}}}{(t_7 - t_4)(t_4 - t_1)(t_1 - z_1)(t_5 - t_2)(t_2 - z_2)(t_6 - t_3)(t_3 - z_3)}$}\\
\\
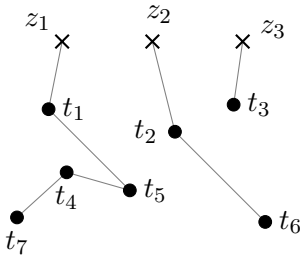
\begin{tikzpicture}[baseline =-5,scale=.6]
\draw[thin, gray] (0,0) -- (-.3,-1.5) -- (1.5, -3.3) -- (.1, -2.9) -- (-1,-3.9);
\draw[thick, black] (-.15, -.15) -- (.15, .15)
			    (-.15, .15) -- (.15, -.15) node[above left=1mm] {$z_1$};
\filldraw [black] (-.3,-1.5) circle (4pt) node[right=.5mm] {$t_1$}
		        (-1,-3.9) circle (4pt) node[below=.5mm] {$t_7$};
\filldraw [black] (.1,-2.9) circle (4pt) node[below=.5mm] {$t_4$};

\draw[thin, gray] (2,0) -- (2.5,-2) -- (4.5, -4);
\draw[thick, black] ( 2 - .15, - .15) -- ( 2 + .15, .15)
			    ( 2 - .15, .15) -- ( 2 + .15, -.15) node[above=2mm] {$z_2$};
\filldraw [black] (2.5,-2) circle (4pt) node[left=1mm] {$t_2$}
		      (1.5,-3.3) circle (4pt) node[right=.5mm] {$t_5$};

\draw[thin, gray] (4,0) -- (3.8,-1.4);
\draw[thick, black] ( 4 - .15, - .15) -- ( 4 + .15, .15)
			    ( 4 - .15, .15) -- ( 4 + .15, -.15) node[above right] {$z_3$};
\filldraw [black] (3.8,-1.4) circle (4pt) node[right=.5mm] {$t_3$}
		      (4.5,-4) circle (4pt) node[right=.5mm] {$t_6$};
\end{tikzpicture}
& \qquad\qquad & \raisebox{-9mm}{$\displaystyle - \frac{F_{c(7)} F_{c(4)} F_{c(5)} F_{c(1)} v_{\lambda^{(1)}} \otimes F_{c(6)} F_{c(2)} v_{\lambda^{(2)}} \otimes F_{c(3)} v_{\lambda^{(3)}}}{(t_7 - t_4)(t_4 - t_5)(t_5 - t_1)(t_1 - z_1)(t_6 - t_2)(t_2 - z_2)(t_3 - z_3)}$}
\end{tabular}
\end{equation*}
The sign in the first term is positive because the lowering operators $F_{c(1)}$, $F_{c(5)}$ and $F_{c(6)}$ appearing in the numerator are ordered in the same way as they appear in the sequence of roots $(t_1, \ldots, t_5, t_6)$. By contrast, the second term depicted above has a minus sign because the order of the operators $F_{c(1)}$ and $F_{c(5)}$ is flipped.
\end{exmp}

\begin{thm}\label{singthm}
The vector $w(\bm{z}, \bm{t})$ belongs to $L(\bm{\lambda})^{\textup{sing}}_{\lambda^{\infty}}$ if the Bethe equations \eqref{BAE} hold.
\begin{proof}
Given $a \in \{ 1, \ldots, n+m-1\}$ we want to show that $E_a w(\bm{z}, \bm{t}) = 0$. 
The strategy of proof is as follows. We first express $E_{a}w(\bm z,\bm t)$ in the form 
\begin{equation}
E_a w(\bm{z}, \bm{t}) = \sum_{\substack{i = 1\\ c(i) = a}}^l \sum_{\bm{n} \in P^i_{l, N}} \kappa_{\bm n, i}(\bm{z}, \bm{t}) (-1)^{|\bm n|} \omega_{\bm n}(\bm{z}, \bm{t}) F_{\bm n} v\label{kappaeqn}
\end{equation}
for certain\footnote{Note that \eqref{kappaeqn} does not by itself uniquely fix these coefficients $\kappa_{\bm n, i}(\bm{z}, \bm{t})$ in general, since the $F_{\bm n} v$ are not in general a linearly independent set of vectors. What matters in the following is that we have a well-defined prescription for computing the $\kappa_{\bm n, i}(\bm{z}, \bm{t})$ such that \eqref{kappaeqn} holds.} coefficients $\kappa_{\bm n, i}(\bm z,\bm t)$, where $P^i_{l, N}$ denotes the set of all ordered partitions of $\{ 1, \ldots, l \} \setminus \{ i \}$ into $N$ parts. We then show that these coefficients are actually zero if the Bethe equations hold. 

To see that $E_aw(\bm z,\bm t)$ can be written in the form \eqref{kappaeqn}, consider the action of $E_a$ on one of the monomials $m_{\bm n}(\bm{z}, \bm{t}) = (-1)^{|{\bm n}|} \omega_{\bm n}(\bm{z}, \bm{t}) F_{\bm n} v$, ${\bm n} \in P_{l ,N}$ appearing in the sum \eqref{weight function}. Since $[E_a, F_b]$ is zero whenever $b\neq a$, we can express $E_a m_{\bm n}(\bm{z}, \bm{t})$ as a sum of the contributions coming from commuting $E_a$ past each of the factors $F_a$ (if any) that appear in $m_{\bm n}$, i.e. a sum over all $i\in\{1,\dots,l\}$ such that $c(i)=a$. Since $[E_a,F_a]=H_a$ and $[H_a,F_b] \propto F_b$, the contribution for a given such $i$ will be proportional to $F_{{\bm n} \setminus \{ i \}} v$, where ${\bm n} \setminus \{ i \}$ is the ordered partition of $\{ 1, \ldots, l \} \setminus \{ i \}$ into $N$ parts obtained from ${\bm n}$ by removing the element $i$.

Now pick and fix a term in \eqref{kappaeqn}, namely choose an $i \in \{ 1, \ldots, l \}$ such that $c(i) = a$ and an ordered partition $\bm n = (n^1_1, \ldots, n^1_{p_1}; \ldots; n^N_1, \ldots, n^N_{p_N})$ of $\{ 1, \ldots, l\} \setminus \{ i \}$ into $N$ parts, so that in particular $p_1 + \ldots + p_N = l - 1$. 
In order to compute $\kappa_{\bm n, i}(\bm{z}, \bm{t})$, consider all the monomials $F_{{\bm m}} v$ with ${\bm m} \in P_{l, N}$ from which $F_{\bm n} v$ can arise under the action of $E_a$. The corresponding ${\bm m} \in P_{l, N}$ are all those obtained by inserting the element $i$ somewhere along the ordered partition $\bm n$. 

Let 
$$\epsilon_i=\prod_{j = 1}^{i-1} (-1)^{|F_{c(j)}| |F_a|}.$$

Let ${\bm n}_0 = (i, n^1_1, \ldots, n^1_{p_1}; \ldots; n^N_1, \ldots, n^N_{p_N}) \in P_{l, N}$ be obtained from insertion of $i$ into the first position of $\bm n$.  We have $(-1)^{|{\bm n}_0|} = (-1)^{|\bm n|} \epsilon_i$

Note that for any ${\bm m} \in P_{l, N}$ obtained from $\bm n$, the sign picked up by bringing $E_a$ through all the generators in front of $F_{c(i)}$ in $F_{\bm n} v$ is equal to $(-1)^{|{\bm n}_0| + |{\bm m}|}$.

Now let $k \in \{ 1, \ldots, N \}$ and suppose ${\bm m} \in P_{l, N}$ is such that $i$ has been inserted in the $k^{\rm th}$ chain of $\bm n$, namely in $(n^k_1, \ldots, n^k_{p_k})$. Then there are three cases to consider: when $i$ is inserted at the end, somewhere in the middle, or at the start of this chain. 
In the first case, the contribution to $\kappa_{\bm n, i}(\bm{z}, \bm{t})$ is given by
\begin{equation*}
\begin{tikzpicture}[baseline =-5,scale=.6]
\draw[thin, gray, dashed] (-2.5,0) -- (-2.5,-2);
\draw[thick, black] (-2.5 - .15, - .15) -- ( -2.5 + .15, .15)
			    (-2.5 - .15, .15) -- ( -2.5 + .15, -.15) node[above=2mm] {$z_k$};
\filldraw [black] (-2.5,-2) circle (4pt) node[left=1mm] {$t_{n^k_1}$};

\draw [<-,very thick] (-1.2,-1.5) -- node[above=1mm]{$E_a$} (0.2,-1.5);

\draw[thin, gray, dashed] (1.5,0) -- (1.5,-2);
\draw[thin, gray] (1.5,-2) -- (2.5, -3.3);
\draw[thick, black] ( 1.5 - .15, - .15) -- ( 1.5 + .15, .15)
			    ( 1.5 - .15, .15) -- ( 1.5 + .15, -.15) node[above=2mm] {$z_k$};
\filldraw [black] (1.5,-2) circle (4pt) node[left=1mm] {$t_{n^k_1}$}
		      (2.5,-3.3) circle (4pt) node[right=.5mm] {$t_i$};
\end{tikzpicture}
\qquad \qquad \qquad
\raisebox{-6mm}{$\displaystyle \epsilon_i \frac{1}{t_i - t_{n^k_1}} \bigg( \lambda^{(k)}_a - \sum_{r = 1}^{p_k} \big( \alpha_{c(n^k_r)}, \alpha^{\vee}_a \big) \bigg),$}
\end{equation*}
where we have used the relations $[H_a, F_b] = - (\alpha_b, \alpha^{\vee}_a) F_b$. Here we have introduced $\lambda^{(k)}_a = (\lambda^{(k)}, \alpha^{\vee}_a)$. The diagram on the left depicts the relevant action of $E_a$ on the $k^{\rm th}$ chain in ${\bm m}$. Similarly, if the element $i$ is inserted somewhere in the middle of the $k^{\rm th}$ chain of ${\bm n}$ then the corresponding coefficient contributing to $\kappa_{\bm n, i}(\bm{z}, \bm{t})$ is given by
\begin{equation*}
\begin{tikzpicture}[baseline =-30,scale=.6]
\draw[thin, gray, dashed] (-2.5,0) -- (-2.5,-2);
\draw[thin, gray] (-2.5,-2) -- (-2.5,-4);
\draw[thin, gray, dashed] (-2.5,-4) -- (-2.5,-5);
\draw[thick, black] (-2.5 - .15, - .15) -- ( -2.5 + .15, .15)
			    (-2.5 - .15, .15) -- ( -2.5 + .15, -.15) node[above=2mm] {$z_k$};
\filldraw [black] (-2.5,-2) circle (4pt) node[left=1mm] {$t_{n^k_{p+1}}$};
\filldraw [black] (-2.5,-4) circle (4pt) node[left=1mm] {$t_{n^k_p}$};

\draw [<-,very thick] (-1.2,-2) -- node[above=1mm]{$E_a$} (0.2,-2);

\draw[thin, gray, dashed] (1.5,0) -- (1.5,-2);
\draw[thin, gray] (1.5,-2) -- (2.7, -3);
\draw[thin, gray] (2.7,-3) -- (1.5, -4);
\draw[thin, gray, dashed] (1.5,-4) -- (1.5,-5);
\draw[thick, black] ( 1.5 - .15, - .15) -- ( 1.5 + .15, .15)
			    ( 1.5 - .15, .15) -- ( 1.5 + .15, -.15) node[above=2mm] {$z_k$};
\filldraw [black] (1.5,-2) circle (4pt) node[below left=-1.5mm] {$t_{n^k_{p+1}}$}
		        (2.7,-3) circle (4pt) node[right=.5mm] {$t_i$}
		        (1.5,-4) circle (4pt) node[left=1mm] {$t_{n^k_p}$};
\end{tikzpicture}\hspace{-1.5cm}
\begin{split} \epsilon_i \frac{t_{n^k_p}-t_{n^k_{p+1}}}{(t_{n^k_p}-t_i)(t_i-t_{n^k_{p+1}})} \bigg( \lambda^{(k)}_a - \sum_{r = p+1}^{p_k} \big( \alpha_{c(n^k_r)}, \alpha^{\vee}_a \big) \bigg) \\ = \epsilon_i \left(\frac{1}{t_{n^k_p} - t_i} + \frac{1}{t_i - t_{n^k_{p+1}}} \right)\bigg( \lambda^{(k)}_a - \sum_{r = p+1}^{p_k} \big( \alpha_{c(n^k_r)}, \alpha^{\vee}_a \big) \bigg).\end{split}
%
\end{equation*}
Finally, if $i$ was inserted at the very start of the $k^{\rm th}$ chain in $\bm n$ then the corresponding contribution to $\kappa_{\bm n, i}(\bm{z}, \bm{t})$ reads
\begin{equation*}
\begin{tikzpicture}[baseline =-25,scale=.6]
\draw[thin, gray] (-2.5,-2) -- (-2.5,-4);
\draw[thin, gray, dashed] (-2.5,-4) -- (-2.5,-5);
\draw[thick, black] (-2.5 - .15, -2 - .15) -- ( -2.5 + .15, -2 + .15)
			    (-2.5 - .15, -2  + .15) -- ( -2.5 + .15, -2 -.15) node[above=2mm] {$z_k$};
\filldraw [black] (-2.5,-4) circle (4pt) node[left=1mm] {$t_{n_{p_k}^k}$};

\draw [<-,very thick] (-1.2,-3.2) -- node[above=1mm]{$E_a$} (0.2,-3.2);

\draw[thin, gray] (1.5,-2) -- (2.7, -3);
\draw[thin, gray] (2.7,-3) -- (1.5, -4);
\draw[thin, gray, dashed] (1.5,-4) -- (1.5,-5);
\draw[thick, black] ( 1.5 - .15, - 2 - .15) -- ( 1.5 + .15, - 2 + .15)
			    ( 1.5 - .15, - 2 + .15) -- ( 1.5 + .15, - 2 -.15) node[above=2mm] {$z_k$};
\filldraw [black] (2.7,-3) circle (4pt) node[right=.5mm] {$t_i$}
		        (1.5,-4) circle (4pt) node[left=1mm] {$t_{n_{p_k}^k}$};
\end{tikzpicture}
\qquad \qquad \qquad
\raisebox{-10mm}{$\displaystyle \epsilon_i \bigg( \frac{1}{t_{n_{p_k}^k} - t_i} + \frac{1}{t_i - z_k} \bigg) \lambda^{(k)}_a.$}
\end{equation*}
Summing up these three contributions to $\kappa_{\bm n, i}(\bm{z}, \bm{t})$ from each $k \in \{ 1, \ldots, N \}$ we obtain
\begin{multline*}
\kappa_{\bm n, i}(\bm{z}, \bm{t}) = \epsilon_i \sum_{k = 1}^N \Bigg( \frac{1}{t_i - t_{n^k_1}} \bigg( \lambda^{(k)}_a - \sum_{r = 1}^{p_k} \big( \alpha_{c(n^k_r)}, \alpha^{\vee}_a \big) \bigg)\\
+ \sum_{p = 1}^{p_k - 1} \bigg(\frac{1}{t_{n^k_p} - t_i} + \frac{1}{t_i - t_{n^k_{p+1}}} \bigg) \bigg( \lambda^{(k)}_a - \sum_{r = p+1}^{p_k} \big( \alpha_{c(n^k_r)}, \alpha^{\vee}_a \big) \bigg)\\
+ \bigg( \frac{1}{t_{n^k_{p_k}} - t_i} + \frac{1}{t_i - z_k} \bigg) \lambda^{(k)}_a \Bigg).
\end{multline*}
Many terms in this sum cancel and after a short calculation we arrive at
\begin{equation*}
\kappa_{\bm n, i}(\bm{z}, \bm{t}) = \epsilon_i \sum_{k = 1}^N \Bigg( - \sum_{p = 1}^{p_k} \frac{\big( \alpha_{c(n_p^{k})}, \alpha^{\vee}_a \big)}{t_i - t_{n^k_p}} + \frac{(\lambda^{(k)}, \alpha^{\vee}_a)}{t_i - z_k} \Bigg).
\end{equation*}
But this last expression coincides exactly with the left hand side of the Bethe equations \eqref{BAE} and therefore vanishes when \eqref{BAE} hold, as required.
\end{proof}
\end{thm}


\subsection{Symmetries of the set of solutions to Bethe equations}\label{secsym}
We let the permutation group 
\be \label{symg}  S_{l_1} \times S_{l_2} \times \dots \times S_{l_{n+m-1}}\ee 
act on the variables $\bm t$ in the natural way: 
for each $1\leq s\leq m+n-1$, an element $\sigma\in S_{l_s}$ permutes the $t_i$ of colour $s$, i.e. those for which $c(i) = s$, and acts trivially on the remaining variables. 
Thus, in the labelling of \eqref{lidef}, $t_{i + \sum_{t<s} l_t} \mapsto t_{\sigma(i) + \sum_{t<s} l_t}$, $1\leq i\leq l_s$, for a given $\sigma\in S_{l_s}$.

This action sends solutions of the Bethe equations \eqref{BAE} to solutions. The value of the weight function \eqref{weight function} at any two solutions in the same orbit yield the same vector up to a sign. We call solutions in the same orbit  \emph{equivalent}.

In the bosonic case (i.e. $\gl(n|0) \cong \gl(n)$) the diagonal entries $(\alpha_{a},\alpha_a)$ of the symmetrized Cartan matrix are all non-zero. 
It follows that Bethe roots $t_i$ of the same colour are prohibited from coinciding by the Bethe equations \eqref{BAE}, and hence the permutation group \eqref{symg} acts freely on the set of solutions to the Bethe equations.

In the general supersymmetric case $\gl(n|m)$, the diagonal entry $(\alpha_{a},\alpha_a)$ of the symmetrized Cartan matrix is zero whenever $\pr{F_a} =1$, so the action of \eqref{symg} is not free in general. However, if $\pr{F_a}=1$ then the weight function is skew-symmetric under the transposition of any two Bethe roots of colour $a$, and is thus actually zero whenever two such roots coincide. Consequently, the group \eqref{symg} does act freely on the set of all those solutions to the Bethe equation for which the corresponding weight function is nonzero.

\subsection{Diagonalization of Hamiltonians in vector representations}\label{vecsec}
For this subsection we specialize to the case in which every site of the spin chain carries the defining representation of $\gl(m|n)$, i.e.  $L(\lambda^{(i)}) = \C^{m|n}$ for all $1\leq i\leq N$. Thus
\be L(\bm \lambda)= \left(\C^{m|n}\right)^{\otimes N}, \qquad v= e_1^{\otimes N}\nn\ee
and 
\be \HH_i = \sum_{\substack{j=1 \\ j\neq i}}^N \frac{\PP_{ij} }{z_i-z_j}, \qquad 1\leq i\leq N,\nn\ee
where, cf \eqref{Gaudin}, 
\be \PP_{ij} = \sum_{a,b=1}^{m+n} e^{(i)}_{ab}\otimes e^{(j)}_{ba} (-1)^{\pr b}\nn\ee
is the graded permutation operator between sites $i$ and $j$. 

\newcommand{\EE}{\mathcal E}
\begin{thm}\label{Hthm}
If the Bethe equations \eqref{BAE} hold, then $w(\bm z,\bm t)$ is an eigenvector of $\HH_i$ with eigenvalue
\be \EE_i:= \sum_{\substack{j=1 \\ j\neq i}}^N \frac 1{z_i-z_j} + \sum_{\substack{j=1\\ c(j)=1}}^{l} \frac 1{t_j-z_i}, \quad\text{for each}\quad 1\leq i\leq N.\nn\ee
\end{thm}
\begin{proof} As in the proof of Theorem \ref{singthm}, we shall first give a prescription for writing $(\HH_i-\EE_i)w(\bm z,\bm t)$ as a linear combination of the vectors $F_{\bm n}v$, ${\bm n}\in P_{l,N}$. There is some freedom in how this is done because these vectors are not linearly independent in general. As in the previous proof, all that  matters is that we have a consistent prescription. We shall then show that with our prescription the coefficients of each $F_{\bm n}v$ is zero if the Bethe equations \eqref{BAE} hold.

The prescription is as follows. Let $\rho_{ij}$ be the involutive map $P_{l,N}\to P_{l,N}$ which exchanges the $i$th and $j$th parts of an ordered partition and leaves the rest unaltered. Then, cf \eqref{weight function}, 
\be \PP_{ij} m_{\bm n}  =  (-1)^{|{\bm n}|} \omega_{\bm n}(\bm{z}, \bm{t}) \PP_{ij} F_{\bm n} v 
                  =  (-1)^{|\rho_{ij}(\bm n)|} \omega_{\bm n}(\bm{z}, \bm{t}) F_{\rho_{ij}(\bm n)} v \nn\ee
and hence
\be \PP_{ij} w(\bm{z}, \bm{t}) = 
    \sum_{{\bm n} \in P_{l, N}} (-1)^{|{\bm n}|} \omega_{\rho_{ij}(\bm n)}(\bm{z}, \bm{t}) F_{\bm n} v
  = \sum_{{\bm n} \in P_{l, N}} \left( (-1)^{|{\bm n}|} \omega_{{\bm n}}(\bm{z}, \bm{t}) F_{\bm n} v \right)  \frac{\omega_{\rho_{ij}(\bm n)}(\bm{z}, \bm{t})}{\omega_{{\bm n}}(\bm{z}, \bm{t})} .\nn\ee 
Thus
\be\label{he} (\HH_i - \EE_i) w(\bm z, \bm t) = 
 \sum_{{\bm n} \in P_{l, N}} \left( (-1)^{|{\bm n}|} \omega_{{\bm n}}(\bm{z}, \bm{t}) F_{\bm n} v \right) 
f(\bm n,i,\bm z, \bm t) \ee
where
\begin{align*}  
f(\bm n,i,\bm z, \bm t ) &:= \sum_{\substack{j=1 \\ j\neq i}}^N
\frac 1 {z_i-z_j} \frac{\omega_{\rho_{ij}(\bm n)}(\bm{z}, \bm{t})}{\omega_{{\bm n}}(\bm{z}, \bm{t})} - \EE_i  \\
&= \sum_{\substack{j=1 \\ j\neq i}}^N
\frac 1 {z_i-z_j} \left(\frac{\omega_{\rho_{ij}(\bm n)}(\bm{z}, \bm{t})}{\omega_{{\bm n}}(\bm{z}, \bm{t})} - 1 \right) 
 - \sum_{\substack{j=1\\ p_j\neq 0}}^N \frac 1{t^{(1)}_j-z_i}.
\end{align*}
Here and in what follows, we use the following useful relabeling of the variables $t$: for any given ordered partition $\bm n$ we shall write
\be t^{(k)}_{i} := t_{n^i_{p_i+1-k}} , \quad 1\leq i\leq N, \, 1\leq k \leq p_i, \nn\ee
so that for example $t^{(1)}_j$ is the Bethe root associated to the first lowering operation in the $j$th tensor factor. (Note that this labeling is dependent on the choice of $\bm n$, i.e. on which term in the sum over ordered partitions we are considering.)

Now we proceed to show that in the sum \eqref{he} the coefficient of each vector $F_{\bm n} v$ is in fact zero. 

Note the structure of the defining representation:
\be \boxed{1}  \xrightarrow{\,\,F_1\,\,} \boxed{2} \xrightarrow{\,\,F_2\,\,} \boxed 3 \xrightarrow{\,\,F_3\,\,} \dots \xrightarrow{F_{m+n-1}} \boxed{m+n}\,,\nn \ee
where $\boxed{k}$ denotes the one-dimensional weight space spanned by $e_k$, $1\leq k\leq m+n$.   
Therefore, in \eqref{weight function}, $F_{\bm n}v$ is zero whenever the number of $F_{a+1}$'s appearing exceeds the number of $F_a$'s, for any $1\leq a<n+m-1$. Hence we can assume the numbers $\bm l= (l_1,\dots,l_r)$, cf. \eqref{lidef}, are weakly decreasing, i.e. form a partition.  

Second, the only ordered partitions $\bm n$ that contribute in the sum \eqref{weight function} are those in which 
$$(c(n^i_1),c(n^i_2),\dots,c(n^i_{p_i})) = (p_i,\dots,3,2,1),$$ for each part $1\leq i\leq N$. And for any such ordered partition, the vector $F_{\bm n}v$ depends on $\bm n$ only through the composition $p_1+\dots+p_N=l$ of $l$: 
\be F_{\bm n} v= F_{p_1} F_{p_1-1} \dots F_1 e_1 \otimes \dots \otimes F_{p_N} F_{p_N-1} \dots F_1 e_1 =e_{p_1+1}\otimes \dots \otimes e_{p_N+1}. \label{Fp} \ee
Let us pick and fix, then, any composition\footnote{In fact the only compositions $\bm p$ we need consider are those consistent with the fixed choice of $\bm l$. That is, for each $1\leq i\leq n+m-1$, $l_i$ must be the number of parts of $p$ of size $\geq i$; or equivalently, the parts of $\bm p$ must be a permutation of the conjugate partition to $\bm l$.} $\bm p=(p_1,\dots,p_N)$ of $l$. Then if $P(\bm p)$ denotes the set of ordered partitions of shape $\bm p$, the coefficient of the vector \eqref{Fp} in \eqref{he} is
\be \sum_{\bm n\in P(\bm p)}  (-1)^{|{\bm n}|} \omega_{{\bm n}}(\bm{z}, \bm{t})f(\bm n,i,\bm z, \bm t) \label{hee}\ee
and we need to show this sum vanishes.

Consider first the case $p_i=0$. Then 
\be f(\bm n,i,\bm z, \bm t) 
= \sum_{\substack{j=1 \\ j\neq i\\p_j= 0}}^N
\frac 1 {z_i-z_j} 
\left(1 - 1 \right) 
+\sum_{\substack{j=1 \\ j\neq i\\p_j\neq 0}}^N
\frac 1 {z_i-z_j} 
\left(\frac{t^{(1)}_j-z_j}{t^{(1)}_j-z_i} - 1 \right) 
 - \sum_{\substack{j=1\\ p_j\neq 0}}^N \frac 1{t^{(1)}_j-z_i} \nn\ee
which vanishes identically, term by term in the sum on $j$. (Note $p_i=0$ implies that actually $j\neq i$ in the final sum here.)

Next consider the case $p_i\geq 1$. Then
\begin{align*} f(\bm n,i,\bm z, \bm t) 
&= \sum_{\substack{j=1 \\ j\neq i\\p_j= 0}}^N
\frac 1 {z_i-z_j} 
\left(\frac{t^{(1)}_i-z_i}{t^{(1)}_i-z_j} - 1 \right) 
+\sum_{\substack{j=1 \\ j\neq i\\p_j\neq 0}}^N
\frac 1 {z_i-z_j} 
\left(\frac{t^{(1)}_i-z_i}{t^{(1)}_i-z_j} \frac{t^{(1)}_j-z_j}{t^{(1)}_j-z_i} - 1 \right) 
 - \sum_{\substack{j=1\\ p_j\neq 0}}^N \frac 1{t^{(1)}_j-z_i} \\
&= \sum_{\substack{j=1 \\ j\neq i\\p_j= 0}}^N
\frac {-1} {t^{(1)}_i-z_j} 
+\sum_{\substack{j=1 \\ j\neq i\\p_j\neq 0}}^N
\frac {t^{(1)}_i-t^{(1)}_j}{(t^{(1)}_i-z_j)(t^{(1)}_j-z_i)}
 - \sum_{\substack{j=1\\ p_j\neq 0}}^N \frac 1{t^{(1)}_j-z_i} \\
&= \sum_{\substack{j=1 \\ j\neq i\\p_j= 0}}^N
\frac {-1} {t^{(1)}_i-z_j} 
-\sum_{\substack{j=1 \\ j\neq i\\p_j\neq 0}}^N
\frac {t^{(1)}_j-z_j}{(t^{(1)}_i-z_j)(t^{(1)}_j-z_i)}
- \frac 1{t^{(1)}_i-z_i} \\
&= \sum_{\substack{j=1 \\ j\neq i}}^N
\frac {-1} {t^{(1)}_i-z_j} 
-\sum_{\substack{j=1 \\ j\neq i\\p_j\neq 0}}^N
\frac {z_i-z_j}{(t^{(1)}_i-z_j)(t^{(1)}_j-z_i)}
- \frac 1{t^{(1)}_i-z_i} \\
&= \sum_{\substack{j=1}}^N
\frac {-1} {t^{(1)}_i-z_j} 
-\sum_{\substack{j=1 \\ j\neq i\\p_j\neq 0}}^N
\frac {z_i-z_j}{(t^{(1)}_i-z_j)(t^{(1)}_j-z_i)}.
\end{align*}
Notice that the first sum here appears in the Bethe equation for $t^{(1)}_i$. 
Consider the second sum. If this sum is empty, i.e. if there are no $j\neq i$ such that $p_j\neq 0$ then one immediately has the equality \eqref{feqn} below, and one can skip to that step. Otherwise, pick any $j\neq i$ such that $p_j\neq 0$. 
Now, taking the sum over $\bm n\in P(\bm p)$ in \eqref{hee} amounts to summing over all permutations of the variables $t^{(k)}_{i}$, $1\leq i\leq N$, among themselves, for each fixed colour $k$. Moreover, the factor $(-1)^{\pr{\bm n}}$ means we should symmetrize for those $k$ such that $\pr{F_k}=0$, and antisymmetrize for those $k$ such that $\pr{F_k}=1$. 
Observe then the following identity valid for complex numbers $z_i\neq z_j$, $t_i\neq t_j$, and $s_i$, $s_j$:
\begin{align}\label{line1} &\frac 1 {s_i-t_i} \frac 1 {t_i - z_i} \frac 1 {s_j-t_j} \frac 1 {t_j - z_j}\times \\
   &\left(  \frac{z_i-z_j}{(t_i-z_j)(t_j-z_i)} + \frac{1\pm 1}{t_i-t_j} + \frac{\mp 1}{t_i-s_j} + \frac{\mp 1 }{s_i-t_j} \mp \frac{t_i-t_j}{(s_i-t_j)(s_j-t_i)} \right)  \pm  (t_i \leftrightarrow t_j) =0.
\nn\end{align} 
By virtue of this we have (the factors in the first line of  \eqref{line1} are present in $\omega_{\bm n}(\bm z,\bm t)$)
\begin{align*} &\sum_{\bm n\in P(\bm p)}  (-1)^{|{\bm n}|} \omega_{\bm n}(\bm z,\bm t) \frac {z_i-z_j}{(t^{(1)}_i-z_j)(t^{(1)}_j-z_i)}\\
=& \sum_{\bm n\in P(\bm p)}  (-1)^{|{\bm n}|} \omega_{\bm n}(\bm z,\bm t)\left( -\frac{(\alpha_1,\alpha_1)}{t^{(1)}_i-t^{(1)}_j} - \frac{(\alpha_1,\alpha_2)}{t^{(1)}_i-t^{(2)}_j} - \frac{(\alpha_1,\alpha_2)}{t^{(2)}_i-t^{(1)}_j} + (-1)^{\pr{F_1}} \frac{t^{(1)}_i-t^{(1)}_j}{(t^{(2)}_i-t^{(1)}_j)(t^{(2)}_j-t^{(1)}_i)}\right),\nn
\end{align*} 
noting $(\alpha_1,\alpha_1) = 1+ (-1)^{\pr{F_1}}$ and $(\alpha_1,\alpha_2)=-(-1)^{\pr{F_1}}$ (recall for us always $\pr{e_1}=0$). The first three terms occur in the Bethe equations for $t^{(1)}_i$ and $t^{(2)}_i$ and we would like to keep them. The final term has the same form as the left-hand side but with $t^{(1)}$'s playing the role of $z$'s and $t^{(2)}$'s playing the role of $t^{(1)}$'s. Thus we may continue to apply the identity above, recursively. If $p_i< p_j$ we need at the final step to use instead the identity
\be\frac 1 {t_i - z_i} \frac 1 {s_j-t_j} \frac 1 {t_j - z_j}
   \left(  \frac{z_i-z_j}{(t_i-z_j)(t_j-z_i)} + \frac{1\pm 1}{t_i-t_j} + \frac{\mp 1}{t_i-s_j} + \frac{\mp 1}{s_j-t_j}  \right) \pm  (t_i \leftrightarrow t_j) =0\nn\ee 
and similarly if $p_i>p_j$.
In this way one eventually obtains
\begin{align}& \sum_{\bm n\in P(\bm p)}  (-1)^{|{\bm n}|} \omega_{{\bm n}}(\bm{z}, \bm{t})f(\bm n,i,\bm z, \bm t)
= \sum_{\bm n\in P(\bm p)}  (-1)^{|{\bm n}|} \omega_{{\bm n}}(\bm{z}, \bm{t}) \label{feqn}\\ 
 &\times \left(  \sum_{\substack{j=1}}^N \frac {-1} {t^{(1)}_i-z_j} 
+\sum_{k=1}^{p_i} \sum_{\substack{j=1\\j\neq i\\p_j\geq k}}^N  \frac{ (\alpha_{k},\alpha_{k})}{t^{(k)}_i-t^{(k)}_j}  
+\sum_{k=1}^{p_i} \sum_{\substack{j=1\\j\neq i\\p_j\geq k+1}}^N  \frac{ (\alpha_{k},\alpha_{k+1})}{t^{(k)}_i-t^{(k+1)}_j}  
+\sum_{k=2}^{p_i} \sum_{\substack{j=1\\j\neq i\\p_j\geq k-1}}^N  \frac{ (\alpha_{k},\alpha_{k-1})}{t^{(k)}_i-t^{(k-1)}_j}  
 \right),\nn\end{align}
noting now that e.g. $(\alpha_2,\alpha_2) = (-1)^{\pr{F_1}} \left(1 + (-1)^{\pr{F_2}} \right)$. 
And this last expression is zero as required, provided the following equations hold:
\begin{align*}   \sum_{\substack{j=1}}^N \frac {-1} {t^{(1)}_i-z_j} 
              + \sum_{\substack{j=1\\j\neq i \\p_j\geq 1}}^N  \frac{ (\alpha_{1},\alpha_{1})}{t^{(1)}_i-t^{(1)}_j}  
              + \sum_{\substack{j=1 \\p_j\geq 2}}^N  \frac{ (\alpha_{1},\alpha_{2})}{t^{(1)}_i-t^{(2)}_j}  &=0\\
                \sum_{\substack{j=1 \\p_j\geq 1}}^N  \frac{ (\alpha_{2},\alpha_{1})}{t^{(2)}_i-t^{(1)}_j}  
              + \sum_{\substack{j=1 \\j\neq i\\p_j\geq 2}}^N  \frac{ (\alpha_{2},\alpha_{2})}{t^{(2)}_i-t^{(2)}_j} 
              + \sum_{\substack{j=1 \\p_j\geq 3}}^N  \frac{ (\alpha_{2},\alpha_{3})}{t^{(2)}_i-t^{(3)}_j}  &=0\\
\vdots\qquad\qquad&\vdots \\
                \sum_{\substack{j=1 \\p_j\geq p_i-1}}^N  \frac{ (\alpha_{p_i},\alpha_{p_i-1})}{t^{(p_i)}_i-t^{(p_i-1)}_j}  
              + \sum_{\substack{j=1 \\j\neq i\\p_j\geq p_i}}^N  \frac{ (\alpha_{p_i},\alpha_{p_i})}{t^{(p_i)}_i-t^{(p_i)}_j} 
              + \sum_{\substack{j=1 \\p_j\geq p_i+1}}^N  \frac{ (\alpha_{p_i},\alpha_{p_i+1})}{t^{(p_i)}_i-t^{(p_i+1)}_j} &=0.\end{align*} 
(The very final sum here is absent if $p_i=n+m-1$. Note also that terms of the form $\frac{(\alpha_k,\alpha_{k\pm 1})}{t^{(k)}_i-t^{(k\pm 1)}_i}$ are present in these equations but cancel correctly to produce the expression above.)
These are indeed the Bethe equations \eqref{BAE} for the variables $t^{(k)}_i$, $k=1,2,\dots, p_i$. The theorem is thus proved. 
\end{proof}

We shall later prove that the same statement is true for tensor products of arbitrary polynomial representations, see Theorem \ref{thm ee2}.

\subsection{Completeness of Bethe Ansatz for $\C^{m|n}\otimes V(\mu)$}
For this section, we pick and fix a polynomial module $V(\mu)$, cf. \S\ref{polynomialmodules}, and consider  the tensor product of $V(\mu)$ with the defining (i.e. first vector) representation $\C^{m|n}$. Theorem \ref{singthm} shows that the weight function $w(z_1,z_2,\bm t)$ is a singular weight vector of $\C^{m|n}\otimes V(\mu)$ if the Bethe equations \eqref{BAE} hold for $z_1,z_2$ and $\bm t$. In the present case we have also the following stronger result, which is usually referred to as \emph{completeness}.
\begin{thm}\label{compmuthm}
For any fixed pair of distinct complex numbers $(z_1,z_2)$, the set of vectors $w(z_1,z_2,\bm t)$, as $\bm t$ runs over the solutions to the Bethe equations \eqref{BAE}, form a basis of $\left(\C^{m|n}\otimes V(\mu)\right)^{\textup{sing}}$. 
\end{thm}
\begin{proof}
Suppose $\mu$ has $k$ parts. Let $\lambda = \sum_{a=1}^{m+n} \lambda_a \eps_a\in \h^*$ be the weight such that $V(\mu) \cong L(\lambda)$, cf. \eqref{VLeqn}. 
Then the Pieri rule, \S\ref{pieri}, states that 
\be \C^{m|n} \otimes V(\mu) \cong \bigoplus_{\text{admissible }a} L\left(\lambda+\eps_1-\sum_{b=1}^{a} \alpha_b\right)\label{csum}\ee
where the sum is over those $a\in \{0,1,2,\dots,m+n-1\}$ such that it is possible to add a box at the end of the row (when $\pr{a+1}=0$) or column (when $\pr{a+1}=1$) corresponding to the $\eps_{a+1}$ component of the integral weight specified by $\mu$. (Cf. the diagrams in Example \ref{diagexmp}.)

In particular, the irreducible component $L(\lambda+\eps_1)$ containing a highest weight vector of $\C^{m|n} \otimes V(\mu)$ is always present and corresponds to the trivial solution (no $t$'s) to the Bethe equations.

Consider, then, any one of the other summands $L\left(\lambda+\eps_1-\sum_{b=1}^a \alpha_b\right)$, i.e. pick an $a$, $1\leq a\leq m+n-1$. If $w(z_1,z_2,\bm t)$ is to yield the highest weight vector of this irreducible component then it must be that $\bm t = (t^{(1)},t^{(2)},\dots,t^{(a)})$ with each $t^{(b)}$, $1\leq b\leq a$, of colour $b$, i.e. associated to $F_b$. 

Without loss of generality, we pick $z_1=1$ and $z_2=0$. 
Then the Bethe equations \eqref{BAE} take the following form.
Note that  $(\alpha_b,\alpha_{b-1}) =- (-1)^{\pr b}$. For each $b$, define
\be \bar\lambda_b := (\lambda,\alpha_b) = (-1)^{\pr b} \lambda_b - (-1)^{\pr{b+1}} \lambda_{b+1}.\nn\ee
Then
\begin{alignat*}{6}
-\frac{1}{t^{(1)} - 1}  
&- \frac {\bar\lambda_1} {t^{(1)}} & &- \frac{(-1)^{\pr 2}}{t^{(1)} - t^{(2)}} & & & & & & & & = 0 \\
&- \frac {\bar\lambda_2} {t^{(2)}}&  &- \frac{(-1)^{\pr 2}}{t^{(2)} - t^{(1)}}  & & - \frac{(-1)^{\pr 3}}{t^{(2)} - t^{(3)}} & & & & & & = 0 \\
& & & & & \vdots & & & & & &  \\
&- \frac {\bar\lambda_{a-1}} {t^{(a-1)}} & & & & & &- \frac{(-1)^{\pr{a-1}}}{t^{(a-1)} - t^{(a-2)}} & &- \frac{(-1)^{\pr{a}}}{t^{(a-1)} - t^{(a)}}  & &   = 0 \\                     
&- \frac {\bar\lambda_a} {t^{(a)}} & &  & & & &  &  &- \frac{(-1)^{\pr a}}{t^{(a)} - t^{(a-1)}}   & &  = 0 .\\
\end{alignat*}
Provided $\bar\lambda_a\neq 0$, the final equation here may be solved for $t^{(a)}$,
\be t^{(a)} = t^{(a-1)} \frac{\bar\lambda_a}{(-1)^{\pr a} + \bar\lambda_a}, \nn\ee
and the solution inserted into the penultimate equation to yield, after some rearrangement,
\be -\frac{\tilde\lambda_{a - 1}}{t^{(a-1)}} +  \frac{(-1)^{\pr{a-1}}}{t^{(a-1)} - t^{(a-2)}} =0,\quad \text{where} \quad \tilde\lambda_{a - 1} = \bar\lambda_a+ \bar\lambda_{a-1} + (-1)^{\pr a} \nn\ee
But this is now of the same form as the equation for $t^{(a)}$. Hence, letting
\be\tilde \lambda_s:= \sum_{r=s}^a \bar\lambda_r + \sum_{r=s+1}^a (-1)^{\pr r}
=  (-1)^{\pr s} \lambda_s - (-1)^{\pr{a+1}} \lambda_{a+1} + \sum_{r=s+1}^a (-1)^{\pr r}
, \quad s=1,2,\dots, a,\nn\ee 
we have by a finite recursion that for each $s=2,3,\dots , a-1$,
\begin{subequations}\label{recrel}
\be t^{(s)} = t^{(s-1)} \frac{\tilde\lambda_s}{(-1)^{\pr s} + \tilde\lambda_s},\ee
and then finally 
\be -\frac 1{t^{(1)}-1} - \frac{\tilde\lambda_1}{t^{(1)}} = 0 \implies t^{(1)} = \frac{\tilde\lambda_1}{1+\tilde\lambda_1}.\ee
\end{subequations}
Thus we have that for $\alpha(\bm l) = \alpha_1+\dots+\alpha_a$ the Bethe equations possess the unique solution 
\begin{align*} t^{(s)} &= \prod_{r=1}^s \frac{\tilde \lambda_r}{(-1)^{\pr r} + \tilde \lambda_r}\\
&= \prod_{r=1}^s \frac{ 
(-1)^{\pr r}\lambda_r - (-1)^{\pr{a+1}} \lambda_{a+1} + \sum_{p=r+1}^a (-1)^{\pr p}}
{  
(-1)^{\pr r}\lambda_r - (-1)^{\pr{a+1}} \lambda_{a+1} + \sum_{p=r}^a (-1)^{\pr p}
}, \quad s=1,2,\dots,a, \end{align*}
provided all denominators and numerators in this product are non-zero, i.e. provided $\tilde \lambda_s\neq 0 \neq (-1)^{\pr s} + \tilde \lambda_s$ for each $s=1,2,\dots,a$.  This is the case, by Lemma \ref{boxlem}. Also, $\tilde\lambda_s\neq (-1)^{\pr s} + \tilde\lambda_{s}$. Consequently the $t^{(s)}$, $1\leq s\leq a$, are all non-zero and $$t^{(a)}\neq t^{(a-1)}\neq \dots \neq t^{(1)} \neq 1.$$
It follows that the weight function $w(1,0;t^{(1)},\dots,t^{(a)})$ is well-defined and non-zero.
\end{proof}

\begin{rem} In the setting above, since there is at most one Bethe root of any given colour, distinct solutions to the Bethe equations are always non-equivalent in the sense of \S\ref{secsym}. 
\end{rem}

\subsection{Simple spectrum of Gaudin Hamiltonians for $\C^{m|n}\otimes V(\mu)$.}
We continue, as in the previous subsection, to specialize to the tensor product of a polynomial module $V(\mu)$ with a copy of the defining representation $\C^{m|n}$. 
\begin{prop}\label{specprop}
Let $w, w' \in \C^{m|n}\otimes V(\mu)$ be two Bethe vectors corresponding to distinct solutions to the Bethe equations. Then $w,w'$ are eigenvectors of the Gaudin Hamiltonian $\HH:= \HH_1 =-\HH_2$ with distinct eigenvalues. 
\end{prop}
\begin{proof}
Recall the relation \eqref{Ceqn} for the quadratic Casimir of \S\ref{casimir}.
Since $\mc C$ acts as a constant in any irreducible module, $\mc C\otimes 1 + 1\otimes \mc C$ acts as a constant on $\C^{m|n} \otimes V(\mu)$. It remains to consider the spectrum of $\Delta \mc C$.
As was shown in the preceding subsection, if $w$ and $w'$ are Bethe vectors corresponding to distinct solutions to the Bethe equations then they are highest weight vectors of two non-isomorphic irreducible components of $\C^{m|n}\otimes V(\mu)$. Specifically, there exist integers $r$ and $a$, $1\leq r\leq a\leq m+n-1$, such that (without loss of generality) $w$ is a highest weight vector of a copy of $L(\lambda+\eps_r)$ in $\C^{m|n}\otimes V(\mu)$ and $w'$ is a highest weight vector of a copy of $L(\lambda+\eps_{a+1})$, cf. \eqref{csum}.   
But now Lemma \ref{clem} and Lemma \ref{boxlem} together imply that the values of the Casimir $\mc C$ on these two irreducibles are distinct, as required. 
\end{proof}
Since the Gaudin Hamiltonians are symmetric with respect to the tensor Shapovalov form, cf. Proposition \ref{Hcomprop}, pairs of eigenvectors with distinct eigenvalues are orthogonal with respect to this form. 



\section{Results from functoriality}\label{func sec}
The results in this section rely on the following well-known lemma from algebraic geometry.
\begin{lem}\label{limlem} Let $n\in \Z_{\geq 1}$ and
suppose $f^{(\vareps)}_k(x_1,\dots,x_n)=0$, $1\leq k\leq n$, is a system of $n$ algebraic equations, depending algebraically on a complex parameter $\vareps$, for $n$ complex variables $x_1,\dots, x_n$. If $(x_1^{(0)},\dots,x_n^{(0)})$ is an isolated solution when $\vareps=0$, then for sufficiently small $\vareps$, there exists an isolated solution $(x_1^{(\vareps)},\dots,x_n^{(\vareps)})$, depending algebraically on $\vareps$, such that
$$x_k^{(\vareps)} = x_k^{(0)} + \mc O(\vareps).$$ 
\qed\end{lem}
\subsection{Completeness of Bethe Ansatz for $\C(m|n)^{\otimes N}$.}
For this subsection we again specialize as in \S\ref{vecsec} to the case in which $L(\bm \lambda)= \left(\C^{m|n}\right)^{\otimes N}$. We shall show the following.
\begin{thm}\label{compthm}
For generic pairwise distinct complex numbers $\bm z = (z_1,\dots,z_N)$, the Gaudin Hamiltonians $(\HH_1,\dots,\HH_N)$ acting in $\left(\left(\C^{m|n}\right)^{\otimes N}\right)^{\textup{sing}}$ are diagonalizable and have joint simple spectrum. Moreover,  for generic $\bs z$
there exists a set of solutions $\{\bs t_i,\ i\in I\}$
of the Bethe equations \eqref{BAE} such that the corresponding Bethe vectors
$\{w(\bm z,\bm t_i),\ i\in I\}$, form a basis of $\left(\left(\C^{m|n}\right)^{\otimes N}\right)^{\textup{sing}}$. 
\end{thm}
\begin{proof}

For each $N\in \Z_{\geq 1}$, 
\be \left(\C^{m|n}\right)^{\otimes N} \cong \bigoplus_{\mu}  V(\mu)^{\oplus m_{\mu,N}}, \label{decomp}\ee 
where the sum is over hook partitions $\mu$ of $N$, and $m_{\mu,N}\in \Z_{\geq 0}$ is the multiplicity of the irreducible polynomial module $V(\mu)$ in this tensor product. Cf. \S\ref{polynomialmodules}. 
The Pieri rule, \S\ref{pieri}, implies that $m_{\mu, N}$ is the number of distinct ways to construct the hook diagram of $\mu$ by adding one box at a time starting from the empty diagram, such that every intermediate step is a legal hook diagram. 
Indeed
\be \left(\C^{m|n}\right)^{\otimes N} \cong \bigoplus_{\mu_1,\dots,\mu_N} V_{\mu_1,\dots,\mu_N}, \nn\ee
where the sum is over all sequences $(\mu_1,\dots,\mu_N)$ of hook partitions such that the diagram of $\mu_k$ has $k$ boxes and is contained, for each $k<N$, in the diagram of $\mu_{k+1}$, and where $V_{\mu_1,\dots,\mu_{N}}\cong V(\mu_N)$ is the irreducible component associated to this sequence. That is, $V_{\mu_1,\dots,\mu_{N}}$ is defined recursively by the demand that for each $2\leq k\leq N$, $V_{\mu_1,\dots,\mu_k}$ is the unique irreducible component of $V_{\mu_1,\dots,\mu_{k-1}}\otimes \C^{m|n}$ isomorphic to $V(\mu_{k})$. 

Pick pair-wise distinct non-zero complex numbers $\tilde z_2,\dots,\tilde z_N$. Theorem \ref{compmuthm} allows us to define, recursively, a highest weight vector $v_{\mu_1,\dots,\mu_N}\in\left(\C^{m|n}\right)^{\otimes N}$ of each of these irreducible components  $V_{\mu_1,\dots,\mu_N}$. Indeed, suppose we have such a set of highest weight vectors $v_{\mu_1,\dots,\mu_{k-1}}\in\left(\C^{m|n}\right)^{\otimes {k-1}} $ for some $2\leq k\leq N$ (and clearly we do in the base case $k=2$). Then $v_{\mu_1,\dots,\mu_k}$ is defined to be the Bethe vector $w(0,\tilde z_k,\bar t^{(1)},\dots,\bar t^{(a)})\in V_{\mu_1,\dots,\mu_{k-1}}\otimes \C^{m|n}$, for the unique $a$ such that the $U(\gl(m|n))$-module through $w$ is isomorphic to $V(\mu_{k})$.

Call these $v_{\mu_1,\dots,\mu_N}\in\left(\C^{m|n}\right)^{\otimes N}$ the \emph{iterated singular vectors}. Associated to each $v_{\mu_1,\dots,\mu_N}$ is the collection 
\be\label{tsformu}
\bar{\bm t}=(\bar t_2^{(1)},\dots,\bar t_2^{(a_2)};\dots; \bar t_N^{(1)},\dots,\bar t_N^{(a_N)})\ee 
consisting of all the Bethe roots used in its construction. Here $a_2,\dots,a_N\in \Z_{\geq 0}$ are the numbers of lowering operators at each step (so actually $a_2\leq 1$), and $\bar t_k^{(b)}$ is of colour $b$, i.e. is associated to $F_b$. For each $k$, the $\bar t_k^{(b)}$, $1\leq b\leq a_k$, solve the Bethe equations for $V_{\mu_1,\dots,\mu_{k-1}}\otimes \C^{m|n}$.

To prove the theorem, we shall show that to each iterated singular vector $v_{\mu_1,\dots,\mu_N}$ in some region of parameters $\bs z$, 
we can associate a Bethe vector $w_{\mu_1,\dots,\mu_N}\in \left(\C^{m|n}\right)^{\otimes N}$ of the same weight, in such a way that these Bethe vectors form a basis. By Theorem \ref{singthm}, these vectors will be singular and by Theorem \ref{Hthm} they will be eigenvectors of the Gaudin Hamiltonians with explicit eigenvalues. We then show that the sets of eigenvalues are different for different constructed Bethe vectors.

To construct the Bethe vector $w_{\mu_1,\dots,\mu_N}$ associated to $v_{\mu_1,\dots,\mu_N}$, 
we must seek a solution to the Bethe equations for   $\left(\C^{m|n}\right)^{\otimes N}$ with Bethe roots 
\be \bm t=( t_2^{(1)},\dots, t_2^{(a_2)};\dots;  t_N^{(1)},\dots, t_N^{(a_N)})\nn\ee
where $t_k^{(b)}$ is of colour $b$, i.e. is associated to $F_b$. 
For these choices of colours, the Bethe equations are:
\begin{subequations}\label{baes}
\be\label{bae1}
\sum_{i=1}^N \frac 1 {t_j^{(1)}-z_i} - \sum_{\substack{i=2\\ i\neq j\\ a_i\geq 1}}^N \frac{(\alpha_1,\alpha_1)}{t_j^{(1)} - t_i^{(1)}} - \sum_{\substack{i=2\\a_i\geq 2}}^n \frac {(\alpha_1,\alpha_2)}{t_j^{(1)} - t_i^{(2)}} = 0 \ee
for all $j$ such that $a_j\geq 1$; and then for all $2\leq b\leq n+m-1$,
\be  - \sum_{\substack{i=2\\ a_i\geq b-1}}^N \frac{(\alpha_b,\alpha_{b-1})}{t_j^{(b)} - t_i^{(b-1)}}
 - \sum_{\substack{i=2\\ i\neq j\\a_i\geq b}}^N \frac{(\alpha_b,\alpha_b)}{t_j^{(b)} - t_i^{(b)}} 
  - \sum_{\substack{i=2\\ a_i\geq b+1}}^n \frac {(\alpha_b,\alpha_{b+1})}{t_j^{(b)} - t_i^{(b+1)}} = 0\label{bae2} \ee
\end{subequations}
for all $j$ such that $a_j\geq b$ (where the final sum is absent in the case $b=m+n-1$).

Pick $z\in \C$, and $\vareps\in \Cx$. Suppose (for now) that 
\be z_1= z ,\qquad\text{and}\qquad z_k = z+ \vareps^{N+1-k} \tilde z_k,\quad 2\leq k\leq N.\label{zt}\ee
Define $\tilde t_k^{(b)}$ by 
\be t_{k}^{(b)} = z + \vareps^{N+1-k} \tilde t_k^{(b)},\quad 2\leq k\leq N,\,\,1\leq b\leq a_k.\label{tt}\ee
Now we consider the leading asymptotic behavior of each of the Bethe equations in turn, as $\vareps\to 0$. 
We claim that at this leading order, the Bethe equations reduce to the equations obeyed by the variables $\bar{\bm t}$.


Consider for example the leading order of the Bethe equation for $t_k^{(1)}$. Note that 
\be\sum_{i=1}^N \frac 1 {t_k^{(1)}-z_i} = \left(\frac{k-1}{\tilde t_k^{(1)}} + \frac 1{\tilde t_k^{(1)} -\tilde z_k}+ \mc O(\vareps)\right) \vareps^{-N-1+k},\nn\ee  
\be  \sum_{\substack{i=2\\ i\neq k\\ a_i\geq 1}}^N \frac{(\alpha_1,\alpha_1)}{t_k^{(1)} - t_i^{(1)}} = 
\left(\sum_{\substack{i<k\\  a_i\geq 1}}\frac{ (\alpha_1,\alpha_1)}{\tilde t_k^{(1)}} + \mc O(\vareps) \right) \vareps^{-N-1+k},\nn\ee
and similarly
\be \sum_{\substack{i=2\\a_i\geq 2}}^n \frac {(\alpha_1,\alpha_2)}{t_j^{(1)} - t_i^{(2)}} =
\left(\sum_{\substack{i<k\\  a_i\geq 2}}\frac{ (\alpha_1,\alpha_2)}{\tilde t_k^{(1)}} + 
                       \frac{ (\alpha_1,\alpha_2) }{t_k^{(1)} - t_k^{(2)}}+  \mc O(\vareps) \right) \vareps^{-N-1+k}.\nn\ee
Let $\lambda_{k-1}\in \h^*$ be the weight such that $L(\lambda_{k-1}) \cong V(\mu_{k-1})$, cf. \eqref{VLeqn}. Then by definition of the numbers $a_i$, we have $$\lambda_{k-1} = (k-1)\eps_1-\sum_{p=1}^{m+n-1}\sum_{\substack{i<k\\a_i\geq p}}\alpha_p$$ and in particular $(\lambda_{k-1},\alpha_1) = k-1-\sum_{\substack{i<k\\  a_i\geq 1}}(\alpha_1,\alpha_1) - \sum_{\substack{i<k\\a_i\geq 2}} (\alpha_2,\alpha_1)$, so we find that 
\be  \frac{(\lambda_{k-1},\alpha_1)}{\tilde t_k^{(1)}} + \frac 1 {\tilde t_k^{(1)} - \tilde z_k} - \frac{\left(\alpha_1,\alpha_2\right)}{\tilde t_k^{(1)} - \tilde t_k^{(2)}} = \mc O(\vareps)\label{leadb}\ee
(or the same equation without the final term if $a_k=1$). And at leading order this is indeed the Bethe equation for $\bar t_{k}^{(1)}$ from the set of Bethe equations for the tensor product $V(\mu_{k-1})\otimes \C^{m|n}$, with the tensor factors assigned to the points $0$ and $\tilde z_k$ respectively.  The other equations work similarly. 

By Lemma \ref{limlem} it follows that for sufficiently small $\vareps$ there must be a solution to the Bethe equations \eqref{baes} of the form $\tilde t_k^{(a)} = \bar t_k^{(a)} + \mc O(\vareps)$. 

Now we show that the set of Bethe vectors $w_{\mu_1,\dots,\mu_N}$ corresponding to such solutions form a basis  of $\left(\left(\C^{m|n}\right)^{\otimes N}\right)^{\textup{sing}}$ for $\vareps$ sufficiently small. For this it suffices to show that $w_{\mu_1,\dots,\mu_N}$ has leading asymptotic behavior \be w_{\mu_1,\dots,\mu_N}=\vareps^K(v_{\mu_1,\dots,\mu_N} + \mc O(\vareps))\label{wasym},\ee 
as $\vareps\to 0$, for some integer power $K$. So consider the definition \eqref{weight function} of $w_{\mu_1,\dots,\mu_N}$. In the sum over ordered partitions, we can isolate those summands in which every factor in the denominator is of the form
\be t^{(a)}_k - t^{(b)}_k \quad\text{or}\quad t^{(a)}_k - z_{k}, \nn\ee
from all the rest (which have factors of the form $t^{(a)}_k-t^{(b)}_l$ or $t^{(a)}_k-z_l$, $k\neq l$). Let $w_{\mu_1,\dots,\mu_N} = w + y$ be the corresponding decomposition of $w_{\mu_1,\dots,\mu_N}$, with $w$ the sum of these distinguished summands. After substitution using \eqref{tt} and \eqref{zt}, one finds that 
\be w = \left(\prod_{k=2}^N \left(\vareps^{-N-1+k}\right)^{a_k} \right) v_{\mu_1,\dots,\mu_N}  \nn\ee
and that $y$ is subleading compared to $w$, which is what we had to show.

Consider two distinct Bethe vectors $w_{\mu_1,\dots,\mu_N}$ and $w_{\mu_1',\dots,\mu_N'}$ constructed as above. By Theorem \ref{Hthm} both are simultaneous eigenvectors of the quadratic Gaudin Hamiltonians $\HH_1,\dots,\HH_N$.  Let $k$, $2\leq k\leq N$, be the smallest such that $\mu_l'=\mu_l$ for all $1\leq l \leq k-1$. (Certainly $\mu_1'=\mu_1$, since both correspond to the unique diagram with a single box.) Consider the Hamiltonian $\HH_k$. When the $z_i$ are chosen as in \eqref{zt} then one finds
\be \HH_k = \vareps^{-N-1+k} \left(\sum_{j=1}^{k-1} \frac{\PP_{jk}}{\tilde z_k} + \mc O(\vareps)\right).\label{Hasym}\ee
Note that $\Delta^{(k-1)}X= \sum_{j=1}^{k-1} X^{(j)}$ for all $X\in \gl(m|n)$. Thus $$\sum_{j=1}^{k-1} \frac{\PP_{jk}}{\tilde z_k} = \frac{ \sum_{a,b}\left(\sum_{j=1}^{k-1} e_{ab}^{(j)}\right) e_{ba}^{(k)}(-1)^{\pr b}}{\tilde z_k}$$ is the image under the embedding $V(\mu_{k-1}) \otimes \C^{m|n}\hookrightarrow (\C^{m|n})^{\otimes k}$ of the quadratic Gaudin Hamiltonian $\HH$ of the spin chain $V(\mu_{k-1})\otimes \C^{m|n}$ with sites at $\tilde z_{k}$ and $0$. Since $\mu_k\neq \mu'_k$, Proposition \ref{specprop} guarantees that $v_{\mu_1,\dots,\mu_k}$ and $v_{\mu'_1,\dots,\mu'_k}$ are eigenvectors of $\HH$ with distinct eigenvalues. It follows that $v_{\mu_1,\dots, \mu_N}$ and $v_{\mu'_1,\dots, \mu'_N}$ are eigenvectors of $\sum_{j=1}^{k-1} \frac{\PP_{jk}}{\tilde z_k}\in \End\left(\left(\C^{m|n}\right)^{\otimes N}\right)$ with distinct eigenvalues. By \eqref{wasym} and \eqref{Hasym}, we have that the eigenvalues of $\HH_k$ on $w_{\mu_1,\dots,\mu_N}$ and $w_{\mu_1',\dots,\mu'_N}$ are 
distinct. 

The argument above establishes that the set of points $\bm z= (z_1,\dots,z_N)$ for which the  Gaudin Hamiltonians are diagonalizable with joint simple spectrum is non-empty. It is a Zariski-open set. Therefore it has full dimension, i.e. the result is true for generic $\bm z$, as required.
\end{proof}

\subsection{Bethe ansatz for arbitrary polynomial representations}
For this subsection we consider the case in which each site of the spin chain carries an arbitrary polynomial representation of $\gl(m|n)$, i.e.
\be L(\lambda^{(i)}) \cong V(\mu^{(i)}),\quad 1\leq i\leq N \nn\ee
for some hook partitions $\mu^{(1)},\dots,\mu^{(N)}$, cf. \S\ref{polynomialmodules}. 
We shall extend Theorem \ref{Hthm} to this case. Since the direct computation used in the proof of Theorem \ref{Hthm} becomes cumbersome in general, we use analytic arguments.

\begin{thm}\label{thm ee2}
If the Bethe equations \eqref{BAE} hold, then $w(\bm z,\bm t)$ is an eigenvector of $\HH_i$ with eigenvalue
\be \EE_i:= \sum_{\substack{j=1 \\ j\neq i}}^N \frac {\left(\lambda^{(i)},\lambda^{(j)}\right)}{z_i-z_j} + \sum_{j=1}^{l} \frac{\left(\lambda^{(i)},\alpha_{c(j)}\right)}{t_j-z_i}, \quad\text{for each}\quad 1\leq i\leq N.\label{ee2}\ee
\end{thm}
\begin{proof}
Pick any one $i$, $1\leq i\leq N$. Let $k_i$ be the number of boxes of the hook partition $\mu^{(i)}$. Consider the Gaudin spin chain with $k_i$ sites, each of which carries the defining representation $\C^{m|n}$ of $\gl(m|n)$:
\be  \underset{k_i}{\underbrace{ \C^{m|n} \otimes \dots \otimes \C^{m|n}} }.\nn\ee 
Let the positions of these sites be $0,\tilde z_{i,2},\dots, \tilde z_{i,k_i}$. 
Now pick hook diagrams $\mu^{(i)}_1,\mu^{(i)}_2,\dots,\mu^{(i)}_{k_i}$ such that $\mu_{k_i}^{(i)}=\mu^{(i)}$ and such that the diagram of $\mu^{(i)}_{k}$ has $k$ boxes and is contained in the diagram of $\mu^{(i)}_{k+1}$ for each $k$, $1\leq k\leq k_i-1$. 
As in \eqref{tsformu} in the proof of the previous result, Theorem \ref{compthm}, to this data is associated an iterated singular vector $$v_i:= v_{\mu^{(i)}_1,\dots,\mu^{(i)}_{k_i}},$$ which is a highest weight vector of an irreducible component of $\left(\C^{m|n}\right)^{\otimes k_i}$ isomorphic to $V(\mu^{(i)})$, and a corresponding solution  
\be {\bm \bar t}=( \bar t_{i,2}^{(1)},\dots, \bar t_{i,2}^{(a_{i,2})};\dots;  \bar t_{i,k_i}^{(1)},\dots, \bar t_{i,k_i}^{(a_{i,k_i})})\nn\ee
to the sequence to two-point Bethe equations.

Now consider the Gaudin spin chain with $\sum_{i=1}^N k_i$ sites, each of which carries the defining representation $\C^{m|n}$ of $\gl(m|n)$. We think of these tensor factors as being  grouped as follows: 
\be \underset{k_1}{\underbrace{ \C^{m|n} \otimes \dots \otimes \C^{m|n}} }\otimes\underset\dots \dots\otimes
 \underset{k_N}{\underbrace{ \C^{m|n} \otimes \dots \otimes \C^{m|n}} },\label{bigc}\ee
and write the positions of the sites as $z_{i,k}$, $1\leq i\leq N$, $1\leq k\leq k_i$.

With the $v_i$ as above, the vector
$v_1\otimes \dots \otimes v_N\in\left(\C^{m|n}\right)^{\sum_{i=1}^Nk_i}$ is singular and, for each $i$, $U(\gl(m|n))\on . v_i \cong V(\mu^{(i)})$, where
\be V(\mu^{(1)})\otimes \dots\otimes V(\mu^{(N)})\label{subm}\ee
is the spin chain we are interested in. 

For the rest of the proof, we pick and fix any solution $\bm z = (z_1,\dots,z_N)$, $\bm t = (t_1,\dots,t_l)$ to the Bethe equations \eqref{BAE} for this tensor product, \eqref{subm}. Let $w$ be the corresponding Bethe vector in the module \eqref{subm}, which we will identify with its image in \eqref{bigc}. 

We consider the spin chain \eqref{bigc}, with 
\be
z_{i,1} := z_i \qquad z_{i,k} := z_i + \vareps^{k_i+1-k} \tilde z_{i,k},\quad 2\leq k\leq k_i.\nn\ee
Introduce variables 
\begin{subequations}\label{baevars}
\be t'_i, \qquad 1\leq i\leq l,\ee
and $t^{(b)}_{i,k}$ and $\tilde t^{(b)}_{i,k}$, where $1\leq b\leq a_{i,k}$, $1\leq k\leq k_i$, $1\leq i \leq N$, related by
\be t^{(b)}_{i,k}  = z_i + \vareps^{k_i+1-k} \tilde t^{(b)}_{i,k}.\ee
\end{subequations}
Let $t^{(b)}_{i,k}$ have colour $b$, and $t'_i$ have the same colour as $t_i$, i.e. $c(t_i)$. 
Consider the leading asymptotic behaviour as $\varepsilon\to 0$ of the Bethe equations for the chain \eqref{bigc} for the full collection of variables $\{t^{(b)}_{i,k}\}\cup \{t'_i\}$. We claim that at leading order in $\vareps$, these Bethe equations yield the following statements:
\begin{enumerate}
\item the $\tilde t^{(b)}_{i,k}$ obey the same equations as the $\bar t^{(b)}_{i,k}$, and 
\item the $t'_i$ obey the same equations as the $t_i$, i.e. the Bethe equations for the spin chain \eqref{subm}. 
\end{enumerate}
The first of these statements is seen by examining the leading order, $\vareps^{-k_i-1+k}$, of the Bethe equation for $t^{(b)}_{i,k}$ as in the previous proof, cf. \eqref{leadb}. 
For the second, note that the Bethe equation for $t'_i$ is 
\be \sum_{j=1}^N \sum_{k=1}^{k_j} \frac{\left(\eps_1,\alpha_{c(t_i)}\right)}{t'_i - z_{j,k}}
- \sum_{i=j}^N \sum_{k=1}^{k_j} \sum_{b=1}^{a_{j,k}} \frac{\left(\alpha_b,\alpha_{c(t_i)}\right)}{t'_i - t^{(b)}_{j,k}}
- \sum_{\substack{j=1\\j\neq i}}^l\frac{\left(\alpha_{c(t_j)},\alpha_{c(t_i)}\right)}{t'_i - t'_j} = 0.\nn\ee
As $\vareps\to 0$, the leading order is $\vareps^0$, and we find
\be \sum_{j=1}^N \frac{k_j \left(\eps_1,\alpha_{c(t_i)}\right)}{t'_i - z_{j}}
- \sum_{j=1}^N \frac{\sum_{k=1}^{k_j} \left(\sum_{b=1}^{a_{j,k}}\alpha_b,\alpha_{c(t_i)}\right)}{t'_i - z_j} 
- \sum_{\substack{j=1\\j\neq i}}^l\frac{\left(\alpha_{c(t_j)},\alpha_{c(t_i)}\right)}{t'_i - t'_j} = \mc O(\vareps).\nn\ee
Since, by definition of the $a_{j,k}$, $k_j\eps_1- \sum_{k=1}^{k_j} \sum_{b=1}^{a_{j,k}}\alpha_b=\lambda^{(j)}$, where $\lambda^{(j)}$ is the highest weight of $V(\mu^{(j)})$, we indeed have
\be \sum_{j=1}^N \frac{\left(\lambda^{(j)},\alpha_{c(t_i)}\right)}{t'_i - z_{j}}
- \sum_{\substack{j=1\\j\neq i}}^l\frac{\left(\alpha_{c(t_j)},\alpha_{c(t_i)}\right)}{t'_i - t'_j} = \mc O(\vareps),\nn\ee
which is part (2) of the claim.

It follows given Lemma \ref{limlem} that for sufficiently small $\vareps$ there is a solution $\{t^{(b)}_{i,k}\}\cup\{t'_i\}$ to the Bethe equations for the spin chain \eqref{bigc} of the form \eqref{baevars} with
\be \tilde t_{i,k}^{(b)} = \bar t_{i,k}^{(b)} + \mc O(\vareps), \quad 
     t'_i = t_i + \mc O(\vareps).\label{ttlim}\ee
Moreover, let $w'$ be the corresponding Bethe vector in \eqref{bigc}. Then (arguing as for \eqref{wasym} in the preceding proof) also $w' = \vareps^K (w+\mc O(\vareps))$ as $\vareps\to 0$.

Now we turn to relating the Gaudin Hamiltonians of the spin chain \eqref{bigc} (call them $\HH_{i,k}$, $1\leq k\leq k_i$, $1\leq i\leq N$) to those of the spin chain \eqref{subm} in which we are interested (call them, as usual, $\HH_i$, $1\leq i\leq N$).
Note that $\Delta^r X = \sum_{k=1}^r X^{(k)}$ for all $X\in \gl(m|n)$. Hence, under the natural embedding of \eqref{subm} into \eqref{bigc},
\be \HH_i 
= \sum_{\substack{j=1 \\j\neq i}}^N \sum_{k=1}^{k_i} \sum_{p=1}^{k_j} \frac{e^{(i,k)}_{ab} e^{(j,p)}_{ba}(-1)^{\pr b}}{z_i-z_j} 
\label{embedH}\ee
where $e^{(i,k)}_{ab}$ denotes the matrix $e_{ab}$ acting in the tensor factor associated to the point $z_{i,k}$, i.e. in tensor factor $\sum_{j<i} k_j + k$. 
Now observe that, for each $i$,
\be \sum_{k=1}^{k_i} \HH_{i,k} = 
\sum_{k=1}^{k_i} \sum_{\substack{p=1\\p\neq k}}^{k_i} \frac{ e^{(i,k)}_{ab} e^{(i,p)}_{ba} (-1)^{\pr b}}{z_{i,k} - z_{i,p}} + 
\sum_{\substack{j=1\\j\neq i}}^N \sum_{k=1}^{k_i} \sum_{p=1}^{k_j} \frac{ e^{(i,k)}_{ab} e^{(j,p)}_{ba} (-1)^{\pr b}}{z_{i,k} - z_{j,p}}.
\nn \ee
The first of the two sums on the right is zero by symmetry of the permutation operator $e^{(i,k)}_{ab} e^{(i,p)}_{ba} (-1)^{\pr b}$. The second is $\HH_i$ at leading order, cf \eqref{embedH}, and thus we have 
\be \sum_{k=1}^{k_i} \HH_{i,k}  = \HH_i + \mc O(\vareps).\nn\ee
Now, by Theorem \ref{Hthm}, 
\be \sum_{k=1}^{k_i} \HH_{i,k} w' = \EE_i' w' \nn\ee where
\be
\EE_i' := 
\sum_{k=1}^{k_i} \sum_{\substack{j=1 \\ j\neq i}}^N \sum_{p=1}^{k_j} \frac 1{z_{i,k}-z_{j,p}} + 
\sum_{k=1}^{k_i} \sum_{\substack{p=1\\p\neq k}}^{k_i} \frac 1{z_{i,k}-z_{i,p}} + 
\sum_{k=1}^{k_i} \sum_{j=1}^N \sum_{\substack{p=1\\ a_{j,p}\geq 1}}^{k_j} \frac 1{t^{(1)}_{j,p} - z_{i,k}} +
\sum_{k=1}^{k_i} \sum_{\substack{j=1\\c(j)=1}}^{l} \frac 1{t'_j-z_{i,k}}.
\label{Ep}\ee
The second sum here is clearly zero. We split the third sum into two pieces:
\be \sum_{k=1}^{k_i} \sum_{j=1}^N \sum_{\substack{p=2\\ a_{j,p}\geq 1}}^{k_j} \frac 1{t^{(1)}_{j,p} - z_{i,k}} =\sum_{k=1}^{k_i} \sum_{\substack{j=1\\j\neq i}}^N \sum_{\substack{p=2\\ a_{j,p}\geq 1}}^{k_j} \frac 1{t^{(1)}_{j,p} - z_{i,k}} 
 + \sum_{k=1}^{k_i} \sum_{\substack{p=2\\ a_{i,p}\geq 1}}^{k_i} \frac 1{t^{(1)}_{i,p} - z_{i,k}}.\label{ep2}\ee
The first of these manifestly has leading behavior $\mc O(\vareps^0)$ in the limit $\vareps \to 0$, as do all the remaining non-zero terms in \eqref{Ep}. At first sight, the second sum in \eqref{ep2} appears to be divergent as $\vareps\to 0$. However, we can re-write it using the Bethe equation for $t_{i,p}^{(1)}$:
\begin{align*} \sum_{k=1}^{k_i} \sum_{\substack{p=2\\ a_{i,l}\geq 1}}^{k_i} \frac 1{t^{(1)}_{i,p} - z_{i,k}} 
&= - \sum_{\substack{p=2\\ a_{i,p}\geq 1}}^{k_i}\sum_{\substack{j=1\\ j\neq i}}^N 
    \sum_{r=1}^{k_j} \frac 1{t^{(1)}_{i,p} - z_{j,r}}
+ \sum_{\substack{p=2\\ a_{i,p}\geq 1}}^{k_i} 
\sum_{\substack{r=2\\ a_{i,r}\geq 1\\ r\neq p}}^{k_i} 
  \frac {(\alpha_1,\alpha_1)}{t^{(1)}_{i,p} - t^{(1)}_{i,r}} 
+ \sum_{\substack{p=2\\ a_{i,p}\geq 1}}^{k_i} 
  \sum_{\substack{j=1\\ j\neq i}}^N 
\sum_{\substack{r=2\\ a_{j,r}\geq 1}}^{k_j}
   \frac {(\alpha_1,\alpha_1)}{t^{(1)}_{i,p} - t^{(1)}_{j,r}}  \\
& {} +  \sum_{\substack{p=2\\ a_{i,p}\geq 1}}^{k_i} 
  \sum_{\substack{j=1}}^N 
\sum_{\substack{r=2\\ a_{j,r}\geq 2}}^{k_j}
   \frac {(\alpha_1,\alpha_2)}{t^{(1)}_{i,p} - t^{(2)}_{j,r}}  
+ \sum_{\substack{p=2\\ a_{i,p}\geq 1}}^{k_i} 
  \sum_{\substack{j=1\\c(j)=1}}^l 
   \frac {(\alpha_1,\alpha_1)}{t^{(1)}_{i,p} - t'_j}  
+ \sum_{\substack{p=2\\ a_{i,p}\geq 1}}^{k_i} 
  \sum_{\substack{j=1\\c(j)=2}}^l 
   \frac {(\alpha_1,\alpha_2)}{t^{(1)}_{i,p} - t'_j} 
.\end{align*}
This expression has a similar structure to \eqref{Ep}: the second sum on the right is zero and the rest are manifestly $\mc O(\vareps^0)$ except for one, which we must split up:
\be   \sum_{\substack{p=2\\ a_{i,p}\geq 1}}^{k_i} 
  \sum_{\substack{j=1}}^N 
\sum_{\substack{r=2\\ a_{j,r}\geq 2}}^{k_j}
   \frac {(\alpha_1,\alpha_2)}{t^{(1)}_{i,p} - t^{(2)}_{j,r}}  
=   \sum_{\substack{p=2\\ a_{i,p}\geq 1}}^{k_i} 
  \sum_{\substack{j=1\\j\neq i }}^N 
\sum_{\substack{r=2\\ a_{j,r}\geq 2}}^{k_j}
   \frac {(\alpha_1,\alpha_2)}{t^{(1)}_{i,p} - t^{(2)}_{j,r}}  
+   \sum_{\substack{p=2\\ a_{i,p}\geq 1}}^{k_i} 
\sum_{\substack{r=2\\ a_{i,r}\geq 2}}^{k_i}
   \frac {(\alpha_1,\alpha_2)}{t^{(1)}_{i,p} - t^{(2)}_{i,r}}  .
\nn\ee
Once again, the first of these is $\mc O(\vareps^0)$, and the second can be re-written using a Bethe equation, this time that for $t_{i,r}^{(2)}$. After completing this recursion on the colour index $b$ of $t_{i,k}^{(b)}$, one arrives at the following leading behavior: 
\begin{align*} \EE_i' &= \sum_{\substack{j=1\\ j\neq i}}^N \frac 1 {z_i-z_j} 
\left( k_ik_j - k_i \sum_{\substack{p=2\\ a_{j,p}\geq 1}}^{k_j} 1 - k_j \sum_{\substack{p=2\\ a_{i,p}\geq 1}}^{k_i} 1    
 + \sum_{a=1}^{m+n-1} \sum_{b=1}^{m+n-1} \sum_{\substack{p=2\\ a_{i,p}\geq a}}^{k_i} \sum_{\substack{r=2\\ a_{j,r}\geq b}}^{k_j} (\alpha_a,\alpha_b) 
 \right) \\
 &{} +  \sum_{\substack{j=1\\ c(j) = 1}}^l \frac 1{t'_j- z_i} 
\left(k_i  - \sum_{a=1}^{m+n-1} \sum_{\substack{p=2\\ a_{i,p}\geq a}}^{k_i} (\alpha_a,\alpha_1)   \right)
+ \sum_{c=2}^{m+n-1} \sum_{\substack{j=1\\ c(j) = c}}^l \frac 1{t'_j- z_i} 
\left(- \sum_{a=1}^{m+n-1} \sum_{\substack{p=2\\ a_{i,p}\geq a}}^{k_i} (\alpha_a,\alpha_c)   \right) + \mc O(\vareps).
\end{align*}
But now, recalling that
\be\lambda^{(j)} = k_j\eps_1 - \sum_{k=2}^{k_j} \sum_{b=1}^{a_{j,k}} \alpha_b 
= k_j\eps_1 - \sum_{b=1}^{m+n-1} \alpha_b \sum_{\substack{k=2\\a_{j,k}\geq b}}^{k_j}1, \nn \ee
we find
\be \EE_i' =   \sum_{\substack{j=1\\ j\neq i}}^N \frac{\left(\lambda^{(i)},\lambda^{(j)}\right)}{z_i-z_j} 
+ \sum_{\substack{j=1}}^l \frac {\left(\lambda^{(i)},\alpha_{c(j)}\right)}{t'_j- z_i}+\mc O(\vareps).\nn\ee 
Recalling \eqref{ttlim}, we therefore have that $\EE_i'\to\EE_i$ as $\vareps\to 0$ with $\EE_i$ as in \eqref{ee2}, which is the required result.
\end{proof}

\section{Features of the Supersymmetric models: some examples in low-rank cases}\label{other sec}

\subsection{The case of $\gl(1|1)$}
The case of $\gl(1|1)$ is the simplest. It is considerably simpler than the case of $\gl(2)$. We illustrate the results of the paper in this case through a direct computation.

Irreducible representations $L(\lambda)$ of $\gl(1|1)$ are parametrised by their highest weights $\lambda = r \epsilon_1 + s \epsilon_2$, cf. \S\ref{fdmod}. We shall write $L(r,s) := L(r \epsilon_1 + s \epsilon_2)$.
If $r \neq -s$ then $L(r,s)$ is two-dimensional and has a basis in which the action is given by the matrices
\be
E_{11}=\left(\begin{matrix} r & 0\\ 0& r-1\end{matrix}\right),\quad  
E_{22}=\left(\begin{matrix} s & 0\\ 0& s+1\end{matrix}\right), \quad
E_{21}=\left(\begin{matrix} 0 & 0\\ 1& 0\end{matrix}\right),\quad
E_{12}=\left(\begin{matrix} 0 & r+s\\ 0& 0\end{matrix}\right).
\nn\ee
In particular, the module $L(r,s)$ with $r \neq - s$ is polynomial if $r\in\Z_{\geq 1}$, $s\in\Z_{\geq 0}$.
When $r = -s$ we get a one-parameter family of one-dimensional modules $L(r, -r)$, of which only the trivial module $L(0,0)$ is polynomial. 

In what follows, we restrict our attention to two-dimensional polynomial modules, namely $L(r, s)$ with $r\in\Z_{\geq 1}$, $s\in\Z_{\geq 0}$.
The tensor product of two such modules decomposes as
\begin{equation} \label{Lrs prod}
L(r, s) \otimes L(r', s') = L(r + r', s + s') \oplus L(r + r' - 1, s + s' + 1).
\end{equation}

Now consider the module $L = L(r_1,s_1)\otimes L(r_2,s_2)\otimes\ldots\otimes L(r_n,s_n)$, where $r_i\in\Z_{\geq 1}$, $s_i\in\Z_{\geq 0}$. Let $r := r_1+\ldots+r_n$ and
\begin{equation*}
h_i := r_i+s_i.
\end{equation*}
We let $v\in L$ denote the tensor product of singular vectors in $L(r_i,s_i)$. By recursive application of the relation \eqref{Lrs prod} it follows that the space of singular vectors $L^{\textup{sing}}$ has dimension $2^{n-1}$. Moreover, the space of singular vectors that are eigenvectors of $E_{11}$ with eigenvalue $r - l$ has dimension $\binom {n-1} l$.

We choose pairwise distinct complex numbers $z_1,\ldots,z_n$. The Bethe equations \eqref{BAE} decouple and reduce to $l$ copies of the single equation
\begin{equation} \label{BAE gl11}
\alpha(t) := \sum_{j=1}^n\frac{h_j}{t-z_j}=0.
\end{equation}
For generic $z_1,\ldots,z_n$, it clearly has $n-1$ distinct solutions $t_1, \ldots, t_{n-1}$.  We assume that we are in such a situation. For example, this happens when all the $z_i$ are real numbers. Note that the $t_i$ are all different from the $z_j$.

Let $E_{ab}(u)$ be the operators acting on $L$, depending on a parameter $u$, given by
\be
E_{ab}(u):=\sum_{j=1}^n \frac{E_{ab}^{(j)}}{u-z_j}.
\nn\ee
Then $E_{11}(u)+E_{22}(u)$ acts on $L$ as multiplication by $\alpha(u)$.

Given a collection ${\bm u} = (u_i)_{i=1}^l$ of $l$ complex numbers, the weight function \eqref{weight function} can be written as
\be
w(\bm z,\bm u) = E_{21}(u_1)E_{21}(u_{2})\dots E_{21}(u_l)v.
\nn\ee
We have $[E_{12},E_{21}(u)]=E_{11}(u)+E_{22}(u)=\alpha(u) \operatorname{Id}$, so that $w(\bm z,\bm u)$ is singular if all the $u_i$ satisfy the Bethe equation \eqref{BAE gl11}, namely $\alpha(u_i) = 0$.
Note, moreover, that $E_{21}(u)E_{21}(v) = - E_{21}(v)E_{21}(u)$ for any $u, v \in \C$ and hence, in particular, $(E_{21}(u))^2 = 0$.

The values of the weight function at the solutions of the Bethe ansatz equations:
\be  \label{basis} E_{21}(t_{i_1})E_{21}(t_{i_2})\dots E_{21}(t_{i_l})v, \quad 1\leq i_1<i_2<\dots <i_l \leq n-1,\quad 0\leq l\leq n-1, \ee
are called Bethe vectors.
By counting, Bethe vectors form a basis of $L^{\textup{sing}}$, provided they are linearly independent. To see that they are linearly independent, we show that with respect to the tensor Shapovalov form they are orthogonal and have nonzero norms. And indeed, we have 
\be \left[ E_{12}(u),E_{21}(w)\right]= -\frac{\alpha(u) - \alpha(w)}{u-w}\operatorname{Id}\quad\text{and hence}\quad 
 \left[E_{12}(u),E_{21}(u)\right]= -\alpha'(u)\operatorname{Id}\nn\ee 
from which it follows that 
\be \langle E_{21}(t_{i_1}) \dots E_{21}(t_{i_l}) v, E_{21}(t_{j_1})\dots E_{21}(t_{j_\ell}) v\rangle = \prod_{k=1}^l \delta_{i_kj_k}(-\alpha'(t_{i_k}))= \prod_{k=1}^l  \delta_{i_kj_k}\sum_{i=1}^n\frac{h_i}{(t_{i_k}-z_i)^2},\label{ips}\ee
so that \eqref{basis} are orthogonal. And since $\alpha(t)\prod_{i=1}^n(t-z_i)$ is a polynomial with simple zeros at the Bethe roots $t_{i_j}$, its derivative, $\alpha'(t) \prod_{i=1}^n(t-z_i) + \alpha(t)\frac{\del}{\del t} \prod_{i=1}^n(t-z_i)$ is nonzero at these roots. Hence, given that $\alpha(t_{i_j})=0$ and $ \prod_{i=1}^n(t_{i_j}-z_i)\neq 0$, it must be that $\alpha'(t_{i_j})\neq 0$, so that the norms are nonzero.

Introduce the master function, cf. \cite{SV}:
\be \Phi(\bm t) :=  \prod_{1\leq i<j\leq n} (z_i-z_j)^{r_ir_j-s_is_j} \prod_{j=1}^l\prod_{i=1}^n (t_i-z_j)^{-h_j}.\nn \ee
We have
\be
\frac{\del\log\Phi}{\del t_i}=-\alpha(t_i),\qquad 
\frac{\del^2\log\Phi}{\del t_i\del t_j} = -\delta_{ij}\alpha'(t_i).\nn
\ee
Therefore, the solutions of the Bethe ansatz equations are critical points of the master function (with distinct coordinates different from $z_i$). Moreover the Hessian matrix is diagonal and
the square of the norm of a Bethe vector is the Hessian determinant of $\log\Phi$ at the corresponding critical point, cf. \cite{MV2,MV1}.

The space of singular vectors $L^{\textup{sing}}$ may be identified with the Grassmann algebra $\Lambda^{n-1}$ in $n-1$ generators $\psi_1, \dots, \psi_{n-1}$ via the identification
\begin{equation} \label{Bethe vect gl11}
\psi_{i_1} \ldots \psi_{i_l} \equiv E_{21}(t_{i_1})\ldots E_{21}(t_{i_l})v, \qquad 0\leq l\leq n-1.
\end{equation}
Furthermore, $[E_{11}, E_{21}(u)] = - E_{21}(u)$ so that \eqref{Bethe vect gl11} is an eigenvector of $E_{11}$ with eigenvalue $r - l$.

Next we check directly that the weight function \eqref{Bethe vect gl11} is an eigenvector of the Gaudin Hamiltonians \eqref{Gaudin}. Consider the operator 
\begin{equation} \label{H gl11}
H(u) := \frac{1}{2}\sum_{a,b=1}^2 E_{ab}(u)E_{ba}(u)(-1)^{|b|}.
\end{equation}
It provides a generating function of the quadratic Gaudin Hamiltonians \eqref{Gaudin} since
\be
H(u)=\frac{1}{2}\sum_{i=1}^n \frac{r_i(r_i-1) - s_i(s_i+1)}{(u-z_i)^2} \operatorname{Id} +\sum_{i=1}^n\frac{1}{u-z_i}{\mathcal H}_i,
\nn\ee
where $r_i(r_i-1) - s_i(s_i+1)$ is the value of the Casimir $\mathcal{C}$ from \S 2.8 on the representation $L(r_i, s_i)$.
To compute the action of \eqref{H gl11} on the singular vector \eqref{Bethe vect gl11} we use the fact that
\begin{equation*}
[H(u), E_{21}(t)] = - \frac{1}{u-t} \big( \alpha(u) E_{21}(t) - \alpha(t) E_{21}(u) \big).
\end{equation*}
In particular, if $t$ is one of the Bethe roots $t_i$ then the second term on the right hand side vanishes by the Bethe equation $\alpha(t_i) = 0$. Hence
\begin{equation*}
[H(u),E_{21}(t_i)] = - \frac{\alpha(u)}{u - t_i}E_{21}(t_i),
\end{equation*}
from which it follows that $H(u)$ preserves the space $L^{\textup{sing}}$. Under the identification of the latter with $\Lambda^{n-1}$, the action of $H(u)$ takes the form
\be
H(u)=-\alpha(u)\sum_{j=1}^{n-1} \frac{1}{u-t_j}\psi_j\partial_{\psi_j}+\beta(u)\operatorname{Id},
\nn\ee
where
\be
\beta(u) = - \sum_{i=1}^n\frac{h_i}{2(u-z_i)^2} + \sum_{i,j=1}^n\frac{r_i r_j - s_i s_j}{2(u-z_i)(u-z_j)}
\nn\ee
is the eigenvalue of $H(u)$ on $v$.
The Bethe vectors $\psi_{i_1}\ldots\psi_{i_l}$, $0\leq l\leq n-1$, are 
therefore eigenvectors for $H(u)$ and hence also for ${\mathcal H}_i$. Specifically, we have
\begin{equation*}
\mathcal{H}_i \psi_{i_1}\ldots\psi_{i_l}=\left( \sum_{\substack{j=1\\ j \neq i}}^n \frac{r_i r_j - s_i s_j}{z_i - z_j} + \sum_{j=1}^l \frac{h_i}{t_{i_j} - z_i} \right) \psi_{i_1}\ldots\psi_{i_l}
\end{equation*}
as in Theorem \ref{thm ee2}.

\subsection{An example of non-polynomial modules}
Throughout \S3--\S5 of this paper we work with polynomial modules of $\gl(m|n)$, which are direct analogues to the finite dimensional modules of $\gl(m)$. Recall that the category of polynomial modules is semi-simple.

The category of \emph{all} finite-dimensional representations of $\gl(m|n)$, $m,n>0$, is not semi-simple. In particular, tensor products of irreducible finite-dimensional $\gl(m|n)$-modules can fail to be fully reducible.
It is interesting to see how the Bethe ansatz behaves in such cases. In this section we give one simple example of this type. 

Let us specialize to $\gl(2|1)$, and choose to work with the Dynkin diagram \be\nn\begin{tikzpicture}[baseline =-3,scale=.6] 
\draw[thick] (1,0) -- (2,0);
\foreach \x in {1,2}
\filldraw[fill=white] (\x,0) circle (2mm);
\draw[thick] (2,0)++(-.15,-.15) -- ++(.3,.3);\draw[thick] (2,0)++(.15,-.15) -- ++(-.3,.3);
\draw[thick] (1,0)++(-.15,-.15) -- ++(.3,.3);\draw[thick] (1,0)++(.15,-.15) -- ++(-.3,.3);
\end{tikzpicture}\ee so that both Chevalley lowering operators $F_1=E_{21}$ and $F_2=E_{32}$ are odd -- cf. Example \ref{exmpgl21} -- 
and the Cartan generators $(H_1,H_2)$ of the subalgebra $\mf{sl}(2|1)\subset \mf{gl}(2|1)$ are 
\be H_1=E_{11}+E_{22} \quad\text{and}\quad H_2= E_{22}+E_{33}.\nn\ee 
The symmetrized Cartan matrix is $\bmx 0 & 1 \\ 1 & 0 \emx$.
Since the rank of $\mf{sl}(2|1)$ is 2, it is easy to picture the $\mf{sl}(2|1)$-weights of finite dimensional representations. We draw weight space as shown below; here the labels $(\lambda_1,\lambda_2)$ are the eigenvalues of $(H_1,H_2)$.
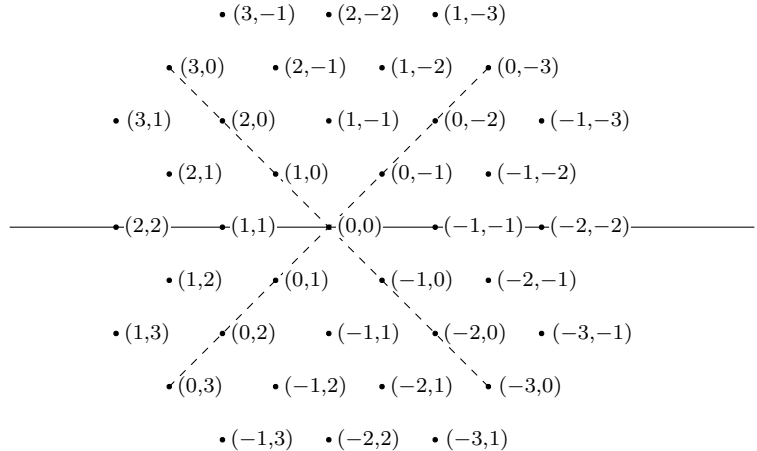
\be\nn\begin{tikzpicture}[rotate=135] 
\draw[dashed] (0,-3) -- (0,3);
\draw[dashed] (-3,0) -- (3,0);
\draw (-4,-4) -- (3,3);

\foreach \x in {-3,3}{\foreach \y in {-1,0,1}{
\filldraw[fill=black] (\x,\y) circle (.3mm) node[right] {$\scriptstyle{(\x,\y)}$};}}
\foreach \x in {-2,2}{\foreach \y in {-2,-1,0,1,2}{
\filldraw[fill=black] (\x,\y) circle (.3mm) node[right=1mm,fill=white, inner sep =0] {$\scriptstyle{(\x,\y)}$};}}
\foreach \x in {-1,1}{\foreach \y in {-3,-2,-1,0,1,2,3}{
\filldraw[fill=black] (\x,\y) circle (.3mm) node[right=1mm,fill=white, inner sep =0] {$\scriptstyle{(\x,\y)}$};}}
\foreach \x in {0}{\foreach \y in {-3,-2,-1,0,1,2,3}{
\filldraw[fill=black] (\x,\y) circle (.3mm) node[right=1mm,fill=white, inner sep =0] {$\scriptstyle{(\x,\y)}$};}}
\end{tikzpicture}\ee
The singlet has highest weight $(0,0)$.
The remaining finite dimensional representations are precisely those with integral highest weights lying strictly in the upper half-plane, i.e. $\lambda_1>\lambda_2$. Among these, the polynomial representations,  \S\ref{polynomialmodules}, are precisely those whose integral highest weights satisfy $\lambda_1>\lambda_2\geq 0$. 
The (Hopf-algebraic) duals of polynomial modules, which also form a semi-simple category, are those finite-dimensional irreducibles with $0\geq \lambda_1>\lambda_2$. 

\begin{exmp} The figures below illustrate the weights, with multiplicities, for the polynomial representations $V(\ydiagram{1})\cong L(1,0,0)$, $V(\ydiagram{2,0})\cong L(2,0,0)$, $V(\ydiagram{1,1}) = L(1,1,0)$ and $V(\ydiagram{2,1}) \cong L(2,1,0)$ respectively.
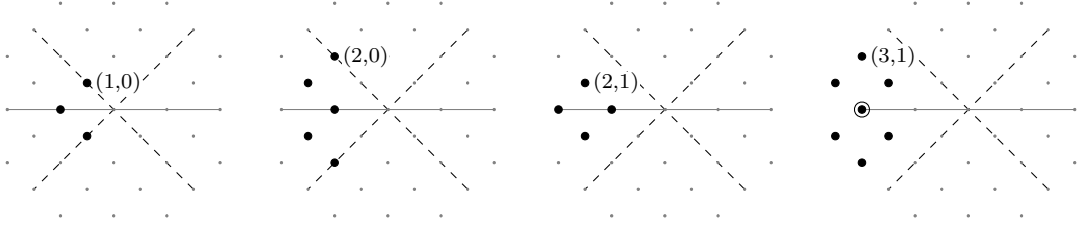
\be\nn\begin{tikzpicture}[rotate=135,scale=.5] 
\draw[dashed] (0,-3) -- (0,3);
\draw[dashed] (-3,0) -- (3,0);
\draw[gray] (-2,-2) -- (2,2);
\foreach \x in {-3,3}{\foreach \y in {-1,0,1}{
\filldraw[gray,fill=gray] (\x,\y) circle (.3mm) ;}}
\foreach \x in {-2,2}{\foreach \y in {-2,-1,0,1,2}{
\filldraw[gray,fill=gray] (\x,\y) circle (.3mm);}}
\foreach \x in {-1,1}{\foreach \y in {-3,-2,-1,0,1,2,3}{
\filldraw[gray,fill=gray] (\x,\y) circle (.3mm) ;}}
\foreach \x in {0}{\foreach \y in {-3,-2,-1,0,1,2,3}{
\filldraw[gray,fill=gray] (\x,\y) circle (.3mm) ;}}
\foreach \x/\y in {1/0,1/1,0/1} {\draw[fill=black] (\x,\y) circle (1mm);}
\node[above,right=1mm,fill=white,inner sep=0] at (1,0) {$\scriptstyle{(1,0)}$};
\end{tikzpicture}\qquad
\begin{tikzpicture}[rotate=135,scale=.5] 
\draw[dashed] (0,-3) -- (0,3);
\draw[dashed] (-3,0) -- (3,0);
\draw[gray] (-2,-2) -- (2,2);
\foreach \x in {-3,3}{\foreach \y in {-1,0,1}{
\filldraw[gray,fill=gray] (\x,\y) circle (.3mm) ;}}
\foreach \x in {-2,2}{\foreach \y in {-2,-1,0,1,2}{
\filldraw[gray,fill=gray] (\x,\y) circle (.3mm); }}
\foreach \x in {-1,1}{\foreach \y in {-3,-2,-1,0,1,2,3}{
\filldraw[gray,fill=gray] (\x,\y) circle (.3mm); }}
\foreach \x in {0}{\foreach \y in {-3,-2,-1,0,1,2,3}{
\filldraw[gray,fill=gray] (\x,\y) circle (.3mm); }}
\foreach \x/\y in {2/0,2/1,1/1,1/2,0/2} {\draw[fill=black] (\x,\y) circle (1mm);}
\node[above,right=1mm,fill=white,inner sep=0] at (2,0) {$\scriptstyle{(2,0)}$};
\end{tikzpicture}\qquad
\begin{tikzpicture}[rotate=135,scale=.5] 
\draw[dashed] (0,-3) -- (0,3);
\draw[dashed] (-3,0) -- (3,0);
\draw[gray] (-2,-2) -- (2,2);
\foreach \x in {-3,3}{\foreach \y in {-1,0,1}{
\filldraw[gray,fill=gray] (\x,\y) circle (.3mm) ;}}
\foreach \x in {-2,2}{\foreach \y in {-2,-1,0,1,2}{
\filldraw[gray,fill=gray] (\x,\y) circle (.3mm); }}
\foreach \x in {-1,1}{\foreach \y in {-3,-2,-1,0,1,2,3}{
\filldraw[gray,fill=gray] (\x,\y) circle (.3mm); }}
\foreach \x in {0}{\foreach \y in {-3,-2,-1,0,1,2,3}{
\filldraw[gray,fill=gray] (\x,\y) circle (.3mm); }}
\foreach \x/\y in {2/1,1/1,1/2,2/2} {\draw[fill=black] (\x,\y) circle (1mm);}
\node[above,right=1mm,fill=white,inner sep=0] at (2,1) {$\scriptstyle{(2,1)}$};
\end{tikzpicture}\qquad
\begin{tikzpicture}[rotate=135,scale=.5] 
\draw[dashed] (0,-3) -- (0,3);
\draw[dashed] (-3,0) -- (3,0);
\draw[gray] (-2,-2) -- (2,2);
\foreach \x in {-3,3}{\foreach \y in {-1,0,1}{
\filldraw[gray,fill=gray] (\x,\y) circle (.3mm) ;}}
\foreach \x in {-2,2}{\foreach \y in {-2,-1,0,1,2}{
\filldraw[gray,fill=gray] (\x,\y) circle (.3mm); }}
\foreach \x in {-1,1}{\foreach \y in {-3,-2,-1,0,1,2,3}{
\filldraw[gray,fill=gray] (\x,\y) circle (.3mm); }}
\foreach \x in {0}{\foreach \y in {-3,-2,-1,0,1,2,3}{
\filldraw[gray,fill=gray] (\x,\y) circle (.3mm); }}
\foreach \x/\y in {3/1,3/2,2/1,2/2,1/2,2/3,1/3} {\draw[fill=black] (\x,\y) circle (1mm);}
\draw (2,2) circle (2mm);
\node[above,right=1mm,fill=white,inner sep=0] at (3,1) {$\scriptstyle{(3,1)}$};
\end{tikzpicture}\ee
The duals of these polynomial modules are the (non-polynomial) finite-dimensional representations $L(0,0,-1)$, $L(0,0,-2)$, $L(0,-1,-1)$ and $L(0,-1,-2)$ respectively. Their weights are shown below.
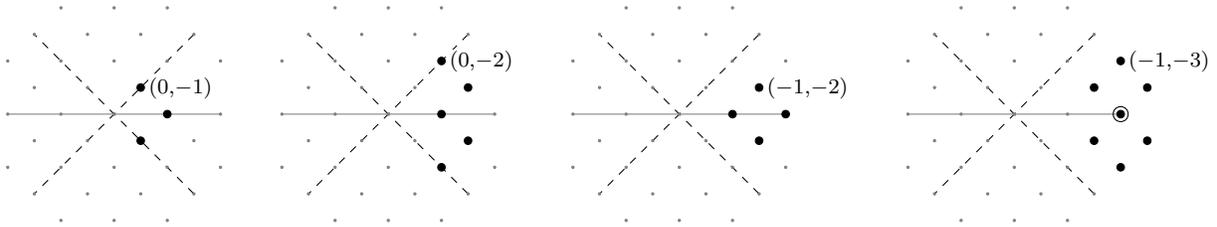
\be\nn\begin{tikzpicture}[rotate=135,scale=.5] 
\draw[dashed] (0,-3) -- (0,3);
\draw[dashed] (-3,0) -- (3,0);
\draw[gray] (-2,-2) -- (2,2);
\foreach \x in {-3,3}{\foreach \y in {-1,0,1}{
\filldraw[gray,fill=gray] (\x,\y) circle (.3mm) ;}}
\foreach \x in {-2,2}{\foreach \y in {-2,-1,0,1,2}{
\filldraw[gray,fill=gray] (\x,\y) circle (.3mm);}}
\foreach \x in {-1,1}{\foreach \y in {-3,-2,-1,0,1,2,3}{
\filldraw[gray,fill=gray] (\x,\y) circle (.3mm) ;}}
\foreach \x in {0}{\foreach \y in {-3,-2,-1,0,1,2,3}{
\filldraw[gray,fill=gray] (\x,\y) circle (.3mm) ;}}
\foreach \x/\y in {1/0,1/1,0/1} {\draw[fill=black] (-\x,-\y) circle (1mm);}
\node[above,right=1mm,fill=white,inner sep=0] at (0,-1) {$\scriptstyle{(0,-1)}$};
\end{tikzpicture}\qquad
\begin{tikzpicture}[rotate=135,scale=.5] 
\draw[dashed] (0,-3) -- (0,3);
\draw[dashed] (-3,0) -- (3,0);
\draw[gray] (-2,-2) -- (2,2);
\foreach \x in {-3,3}{\foreach \y in {-1,0,1}{
\filldraw[gray,fill=gray] (\x,\y) circle (.3mm) ;}}
\foreach \x in {-2,2}{\foreach \y in {-2,-1,0,1,2}{
\filldraw[gray,fill=gray] (\x,\y) circle (.3mm); }}
\foreach \x in {-1,1}{\foreach \y in {-3,-2,-1,0,1,2,3}{
\filldraw[gray,fill=gray] (\x,\y) circle (.3mm); }}
\foreach \x in {0}{\foreach \y in {-3,-2,-1,0,1,2,3}{
\filldraw[gray,fill=gray] (\x,\y) circle (.3mm); }}
\foreach \x/\y in {2/0,2/1,1/1,1/2,0/2} {\draw[fill=black] (-\x,-\y) circle (1mm);}
\node[above,right=1mm,fill=white,inner sep=0] at (0,-2) {$\scriptstyle{(0,-2)}$};
\end{tikzpicture}\qquad
\begin{tikzpicture}[rotate=135,scale=.5] 
\draw[dashed] (0,-3) -- (0,3);
\draw[dashed] (-3,0) -- (3,0);
\draw[gray] (-2,-2) -- (2,2);
\foreach \x in {-3,3}{\foreach \y in {-1,0,1}{
\filldraw[gray,fill=gray] (\x,\y) circle (.3mm) ;}}
\foreach \x in {-2,2}{\foreach \y in {-2,-1,0,1,2}{
\filldraw[gray,fill=gray] (\x,\y) circle (.3mm); }}
\foreach \x in {-1,1}{\foreach \y in {-3,-2,-1,0,1,2,3}{
\filldraw[gray,fill=gray] (\x,\y) circle (.3mm); }}
\foreach \x in {0}{\foreach \y in {-3,-2,-1,0,1,2,3}{
\filldraw[gray,fill=gray] (\x,\y) circle (.3mm); }}
\foreach \x/\y in {2/1,1/1,1/2,2/2} {\draw[fill=black] (-\x,-\y) circle (1mm);}
\node[above,right=1mm,fill=white,inner sep=0] at (-1,-2) {$\scriptstyle{(-1,-2)}$};
\end{tikzpicture}\qquad
\begin{tikzpicture}[rotate=135,scale=.5] 
\draw[dashed] (0,-3) -- (0,3);
\draw[dashed] (-3,0) -- (3,0);
\draw[gray] (-2,-2) -- (2,2);
\foreach \x in {-3,3}{\foreach \y in {-1,0,1}{
\filldraw[gray,fill=gray] (\x,\y) circle (.3mm) ;}}
\foreach \x in {-2,2}{\foreach \y in {-2,-1,0,1,2}{
\filldraw[gray,fill=gray] (\x,\y) circle (.3mm); }}
\foreach \x in {-1,1}{\foreach \y in {-3,-2,-1,0,1,2,3}{
\filldraw[gray,fill=gray] (\x,\y) circle (.3mm); }}
\foreach \x in {0}{\foreach \y in {-3,-2,-1,0,1,2,3}{
\filldraw[gray,fill=gray] (\x,\y) circle (.3mm); }}
\foreach \x/\y in {3/1,3/2,2/1,2/2,1/2,2/3,1/3} {\draw[fill=black] (-\x,-\y) circle (1mm);}
\draw (-2,-2) circle (2mm);
\node[above,right=1mm,fill=white,inner sep=0] at (-1,-3) {$\scriptstyle{(-1,-3)}$};
\end{tikzpicture}\ee
(In both sets of pictures, the two modules on the left are examples of \emph{atypical} representations while the two on the right are \emph{typical}; for the definitions see e.g. \cite{ChengWang}.) 
\end{exmp}

Let us now consider the tensor product
\be L(1,0,0) \otimes L(0,-1,-1),\label{tp}\ee
whose $\sl(2|1)$-weights are shown below.
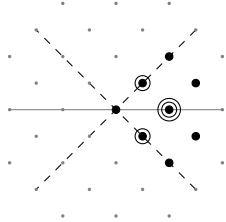
\be\nn\begin{tikzpicture}[rotate=135,scale=.5] 
\draw[dashed] (0,-3) -- (0,3);
\draw[dashed] (-3,0) -- (3,0);
\draw[gray] (-2,-2) -- (2,2);
\foreach \x in {-3,3}{\foreach \y in {-1,0,1}{
\filldraw[gray,fill=gray] (\x,\y) circle (.3mm) ;}}
\foreach \x in {-2,2}{\foreach \y in {-2,-1,0,1,2}{
\filldraw[gray,fill=gray] (\x,\y) circle (.3mm); }}
\foreach \x in {-1,1}{\foreach \y in {-3,-2,-1,0,1,2,3}{
\filldraw[gray,fill=gray] (\x,\y) circle (.3mm); }}
\foreach \x in {0}{\foreach \y in {-3,-2,-1,0,1,2,3}{
\filldraw[gray,fill=gray] (\x,\y) circle (.3mm); }}
\foreach \x/\y in {2/0,2/1,1/1,1/2,0/2,1/0,0/1,0/0} {\draw[fill=black] (-\x,-\y) circle (1mm);}
\foreach \x/\y in {1/1,1/0,0/1} {\draw (-\x,-\y) circle (2mm);}
\draw (-1,-1) circle (3mm);
\end{tikzpicture}\ee
Pick highest weight vectors  $v\in L(1,0,0)$ and $w\in L(0,-1,-1)$. Then the vector
\be u:= E_{21} (v\otimes w) = E_{21} v \otimes w + v\otimes E_{21} w \nn\ee
is singular, for one has $E_{12} E_{21} (v\otimes w) = (E_{11} + E_{22}) (v\otimes w) = 0$. Moreover $u$ is annihilated by $E_{21}$, since $E_{21}^2=0$. Hence the submodule of \eqref{tp} generated by $u$, call it $U$, is isomorphic to $L(0,0,-1)$. And the submodule through $v\otimes w$, call it $X$, is an indecomposable with two composition factors: the submodule $U\subset X$ and the quotient $X/U \cong L(1,-1,-1)$. (The latter is isomorphic to $L(0,0,-2)$ as an $\sl(2|1)$-module.) 
Next, note that the vector $z:= (E_{21}v\otimes E_{21} w)$ spans the one-dimensional weight space of weight $(-1,1,-1)$, and that $u=E_{12} z$. So the submodule of \eqref{tp} generated by $z$, call it $Z$, is another indecomposable with two composition factors: the submodule $U\subset Z$ and the quotient $Z/U\cong L(-1,1,-1)$, which is a singlet. 
Finally, let $y:= (E_{21} v)\otimes w$. Then $v\otimes w = E_{12} y$ and $z = -E_{21} y$, from which we conclude that \eqref{tp} is  cyclic on $y$, and that the irreducible quotient of \eqref{tp} is another copy of $L(0,0,-1)$.
Therefore \eqref{tp} has composition factors $L(1,-1,-1)$, $L(0,0,-1)$, $L(0,0,-1)$ and $L(-1,1,-1)$  (whose dimensions are $5,3,3,1$) and the  structure of the representation can be summarized by the poset 
\be \begin{tikzpicture}    
\matrix (m) [matrix of math nodes, 
row sep=1em,    
column sep=-2em
]    
{           & L(0,0,-1)      & \\    
  L(1,-1,-1)  &     & L(-1,1,-1)   \\    
            & L(0,0,-1)      & \\   };    
\path[thick,->]    
(m-1-2) edge (m-2-1)     
(m-2-1) edge (m-3-2)    
(m-1-2) edge (m-2-3)
(m-2-3) edge (m-3-2);    
\end{tikzpicture}    
.\nn  \ee
We now consider the Bethe ansatz equations \eqref{BAE} and Bethe vectors.
The trivial solution ($l_1=0$, $l_2=0$) corresponds to the singular vector $v\otimes w$. Consider $l_1=1, l_2=0$, i.e. solutions corresponding to applying $F_1=E_{21}$ once to $v\otimes w$. The single Bethe equation is then 
\be \frac{1}{t-z_1} + \frac{0}{t-z_2} = 0\nn\ee
and has no solution. This is despite the fact that there is a singular vector, namely $u$, with the correct weight. 

Thus, the Bethe vectors do not span the space of all singular vectors of the module \eqref{tp}. This is not unexpected, since the module \eqref{tp} is not fully reducible, so not all singular vectors are highest weight vectors of direct summands; in particular $u$ is a descendant of $v\otimes w$. 

Since the module \eqref{tp} is actually indecomposable, and since the Gaudin Hamiltonian $\HH_1=-\HH_2$ commutes with the $\gl(2|1)$-action (Proposition \ref{Hcomprop} part ii), $\HH_1$ has only one generalized eigenvalue. By considering for example the vector $v\otimes w$ it is easy to see that this eigenvalue is actually zero. That is, $\HH_1$ acts nilpotently on the module \eqref{tp}. However, $\HH_1$ has non-trivial Jordan blocks. Indeed, in the basis $u,y$ of the weight space $(0,0,-1)$, one finds that
\be \HH_1 \bmx u & y \emx = \bmx u & y \emx \bmx 0 & 1 \\ 0 & 0 \emx. \nn\ee 

\subsection{On one-point solutions of the Bethe ansatz equations for $\gl(2|1)$}
We continue to work with $\gl(2|1)$, with the conventions of the preceding subsection.
Let us consider the case of a spin chain with one site, at position $z\in \C$. Recall that in the even case of $\gl(n)$ in the one point model, the solution of the Bethe ansatz form a full flag variety of $\operatorname{GL}_n$, \cite{MVcrit}, \cite{MV one}. Therefore it is interesting to see what happens in the supersymmetric case.

In contrast to for example the case of $\gl(3)$, the diagonal entries of the Cartan matrix are zero in the present case. Thus terms of the form e.g. $2/(t_i-t_j)$ are absent from the Bethe equations and so in principle there can be solutions with coincident Bethe roots of the same colour. But, cf. \S\ref{secsym}, for any such solution, the corresponding weight function is identically zero. 

The weight function in our case simplifies as follows.
\begin{lem}\label{gl21wf}
In the case (of $\gl(2|1)$ with Dynkin diagram \begin{tikzpicture}[baseline =-3,scale=.6] 
\draw[thick] (1,0) -- (2,0);
\foreach \x in {1,2}
\filldraw[fill=white] (\x,0) circle (2mm);
\draw[thick] (2,0)++(-.15,-.15) -- ++(.3,.3);\draw[thick] (2,0)++(.15,-.15) -- ++(-.3,.3);
\draw[thick] (1,0)++(-.15,-.15) -- ++(.3,.3);\draw[thick] (1,0)++(.15,-.15) -- ++(-.3,.3);
\end{tikzpicture}, and a chain with one site, at the point $z\in \C$), the weight function \eqref{weight function} reduces to the following. If $l_1=l_2=:l$ then
\be w(z; t_1,\dots, t_l;s_1,\dots,s_l) =  
\frac{\prod_{1\leq i<j\leq l}(s_i-s_j)(t_i-t_j)}{\prod_{i,j=1}^l(s_i-t_j)}
\left(\frac{(F_1F_2)^l v}{\prod_{i=1}^l(s_i-z)} + \frac{(F_2F_1)^l v}{\prod_{i=1}^l(t_i-z)} 
\right).\ee
If $l_1-1=l_2=:l$ then
\be w(z; t_1,\dots, t_{l+1};s_1,\dots,s_{l}) = 
 \frac{\prod_{1\leq i<j\leq l-1}(s_i-s_j)
       \prod_{1\leq i<j\leq l}(t_i-t_j)}
{\prod_{i=1}^l\prod_{j=1}^{l-1}(t_i-s_j)\prod_{i=1}^l(t_i-z)}
(F_1F_2)^l F_1 v;\ee
if $l_2-1=l_1=:l$ the formula is the same with $s\leftrightarrow t$ and $F_1\leftrightarrow F_2$.
In all other cases, $w=0$.
\end{lem}
\begin{proof}
Since $F_1^2=F_2^2=0$ only the monomials shown survive in the weight function \eqref{weight function}. 
So the lemma is equivalent to the following identities: 
\be \sum_{\substack{\bm n\in S_l\\ \bm p\in S_l}} 
\frac{(-1)^{|\bm n|}(-1)^{|\bm p|}}{(t_{p(1)} - s_{n(1)})(s_{n(1)}-t_{p(2)}) \dots (t_{p(l)} - s_{n(l)})(s_{n(l)} - z)}  =
 \frac{\prod_{1\leq i<j\leq l}(s_i-s_j)(t_i-t_j)}{\prod_{i,j=1}^l(s_i-t_j)\prod_{i=1}^l(s_i-z)}.
\label{i1}\ee
and
\be   
\sum_{\substack{\bm n\in S_{l-1}\\ \bm p\in S_l}} 
\frac{(-1)^{|\bm n|}(-1)^{|\bm p|}}{(t_{p(1)} - s_{n(1)})(s_{n(1)}-t_{p(2)}) \dots (t_{p(l)} - z)}  =
 \frac{\prod_{1\leq i<j\leq l-1}(s_i-s_j)
       \prod_{1\leq i<j\leq l}(t_i-t_j)}
{\prod_{i=1}^l\prod_{j=1}^{l-1}(t_i-s_j)\prod_{i=1}^l(t_i-z)}.
\label{i2}\ee
These identities are proved by standard counting of zeroes and poles.
\end{proof}
Let us call a solution to the Bethe equations \emph{admissible} if there are no coincident Bethe roots of the same colour. 

Suppose without loss of generality that $z=0$. 
We pick integers $r_1$ and $r_2$ such that $r_1>r_2>0$. To the single site at the origin we assign the polynomial representation $L(r_1-r_2,r_2,0)$. That is, the vacuum state $v$ has $H_1v=r_1v$, $H_2v =r_2v$. Such polynomial representations, with $r_2>0$, are called \emph{typical} (those with $r_2=0$ are \emph{atypical}).   

In this subsection we identify all admissible solutions to the Bethe equations in typical case, for all values of $l_1$ and $l_2$.
The Bethe equations are
\be -\frac{r_1}{t_i} + \sum_{j=1}^{l_2} \frac 1{t_i-s_j} = 0, \quad 1\leq i\leq l_1\qquad\text{and}\qquad
+\frac{r_2}{s_j} + \sum_{i=1}^{l_1}\frac 1{s_j-t_i} = 0, \quad 1\leq j \leq l_2.\label{gl21bae}\ee
(The sign in the second equation is from $(\alpha_2,\eps_2) = (\eps_2,\eps_2) = -1$.)

Let $y_1(u)=\prod_{i=1}^{l_1}(u-t_i)$ and $y_2(u)=\prod_{i=1}^{l_2}(u-s_i)$ be the polynomials in a variable $u$. Then $y_1$ and $y_2$ are relatively prime and do not vanish at $u=0$. 
\begin{lem} Admissible solutions of the equations (\ref{gl21bae}) exist if and only if $l_1=l_2=l$ and
$l=r_1-r_2$ or $l=0$. In the case $l=0$ the solution is trivial. In the case $l=r_1-r_2$ the polynomials which correspond to the admissible solutions are
\be
y_1=u^l-r_2c, \qquad y_2=u^l-r_1c, \label{gl11pop}
\ee
where $c\in\C^\times$ is an arbitrary non-zero constant.
\end{lem}
\begin{proof}
The equations (\ref{gl21bae}) are equivalent to:
\begin{eqnarray*}
y_1 \quad {\rm divides} \quad u^{r_1+1}(u^{-r_1}y_2)'=-r_1y_2+ uy_2', \\ 
y_2 \quad {\rm divides} \quad u^{-r_2+1}(u^{r_2}y_1)'=r_2y_1+uy_1',
\end{eqnarray*}
where the prime denotes the derivative with respect to $u$.
It follows that $y_1$ and $y_2$ have the same degree $l_1=l_2=l$ and moreover,
\begin{eqnarray*}
(l-r_1)y_1 =-r_1y_2+ uy_2', \\
(l+r_2)y_2 =r_2y_1+uy_1'.
\end{eqnarray*}
Subtracting we obtain 
\be
(l-r_1+r_2)(y_1-y_2)=u(y_1-y_2)'
\ee
It follows that $y_1=y_2+au^{l-r_1+r_2}$ for some constant $a\in\C$. 
Substitute and solve the equation for $y_2$. Recalling that 
$y_2$ is a monic polynomial which does not vanish at zero, we obtain the lemma.
\end{proof}

For the solution described in the lemma, the corresponding Bethe vector is easily computed from Lemma \ref{gl21wf}. Namely, it is a nonzero multiple of the vector
\be w := \left(r_1(F_1F_2)^l + r_2(F_2F_1)^l\right)v.\ee
Note that the $\sl(2|1)$ weight is given by $H_1w=r_2w$, $H_2w=r_1w$.
By direct calculation one verifies that $w$ is a singular vector of the Verma module generated by $v$. In fact, if $M_{p,q}$ denotes the Verma module generated by a highest weight vector $v_{p,q}$ with $H_1v_{p,q}=pv_{p,q}$ and $H_2v_{p,q} = qv_{p,q}$  then one has $L(r_1-r_2,r_2,0) \cong M_{r_1,r_2} \big/ M_{r_2,r_1}$ whenever, as here, $L(r_1-r_2,r_2,0)$ is a typical polynomial module. 

In this case, the solutions of the Bethe equations are in 1-1 correspondence with the solutions of the corresponding Bethe equations for $\gl(2)$, which form a flag variety in a two dimensional space \cite{MVcrit}.

\end{document}